\documentclass[screen,sigconf]{acmart}
\usepackage{subcaption}
\usepackage{mathtools}
\usepackage{tikz}
\usetikzlibrary{arrows.meta,automata,calc}
\usetikzlibrary{decorations.pathreplacing,decorations.pathmorphing}
\usetikzlibrary{positioning}
\usetikzlibrary{shapes.geometric}
\usetikzlibrary{patterns}
\usetikzlibrary{intersections}
\usetikzlibrary{angles}
\usepackage{cleveref}
\usepackage{enumitem}
\usepackage{ifthen}

\newboolean{fullversion}
\setboolean{fullversion}{true}

\copyrightyear{2021}
\acmYear{2021}
\setcopyright{licensedothergov}\acmConference[HSCC '21]{24th ACM International Conference on Hybrid Systems: Computation and Control}{May 19--21, 2021}{Nashville, TN, USA}
\acmBooktitle{24th ACM International Conference on Hybrid Systems: Computation and Control (HSCC '21), May 19--21, 2021, Nashville, TN, USA}
\acmPrice{15.00}
\acmDOI{10.1145/3447928.3456705}
\acmISBN{978-1-4503-8339-4/21/05}

\newcommand{\RR}{\mathbb{R}}
\newcommand{\RRtarski}{\mathfrak{R}_0}
\newcommand{\RRexp}{\mathfrak{R}_{\exp}}
\newcommand{\RRexpsin}{\mathfrak{R}_{\exp,\sin}}
\newcommand{\RRtheory}{\mathfrak{R}}
\newcommand{\CC}{\mathbb{C}}
\newcommand{\QQ}{\mathbb{Q}}
\newcommand{\QQbar}{\overline{\mathbb{Q}}}
\newcommand{\NN}{\mathbb{N}}
\newcommand{\ZZ}{\mathbb{Z}}

\newcommand{\ie}{\emph{i.e.}}
\newcommand{\Reach}{\mathcal{R}}
\newcommand{\Control}{\mathcal{C}}
\newcommand{\Spectrum}{\sigma}
\DeclareMathOperator{\diag}{diag}
\DeclareMathOperator{\sgn}{sgn}
\DeclareMathOperator{\argmax}{arg\,max}
\newcommand{\dd}{\operatorname{d}\hspace{-0.3ex}}
\DeclareMathOperator{\Span}{span}

\DeclarePairedDelimiter{\Angle}{\langle}{\rangle}
\DeclarePairedDelimiter{\set}{\{}{\}}

\DeclarePairedDelimiter{\twonorm}{\lVert}{\rVert}

\definecolor{darkgreen}{rgb}{0.1,0.7,0.1}

\theoremstyle{acmplain}
\newtheorem{theorem}{Theorem}[section]

\newtheorem{proposition}[theorem]{Proposition}
\newtheorem{lemma}[theorem]{Lemma}

\theoremstyle{acmdefinition}

\newtheorem{definition}[theorem]{Definition}

\theoremstyle{plain}

\newcommand{\apxref}[2]{\ifthenelse{\boolean{fullversion}}{#1}{#2}}

\begin{document}

\title{On the Decidability of Reachability in Continuous Time Linear Time-Invariant Systems}
\apxref{}{
    \titlenote{Full version available at \url{https://arxiv.org/abs/2006.09132}. Proofs marked as }
}

\author{Mohan Dantam}
\email{saitejadms@gmail.com}
\author{Amaury Pouly}
\email{amaury.pouly@irif.fr}
\orcid{0000-0002-2549-951X}
\affiliation{\institution{Université de Paris, IRIF, CNRS} \postcode{F-75013} \city{Paris} \country{France}}

\ccsdesc[500]{Computing methodologies~Computational control theory}

\keywords{LTI systems, control theory, reachability, decidability, linear differential equation, theory of the reals, exponential}

\begin{abstract}
    We consider the decidability of state-to-state reachability in linear
    time-invariant control systems over continuous time.  We analyze this
    problem with respect to the allowable control sets, which are assumed to be the image under
    a linear map of the unit hypercube (\emph{i.e.} zonotopes).
    This naturally models bounded (sometimes called saturated) controls.
    Decidability of the version of the reachability problem in which
    control sets are affine subspaces of $\RR^n$ is a fundamental result in
    control theory. Our first result is decidability in two dimensions ($n=2$) if
    matrix $A$ satisfies some spectral conditions and conditional decidablility in general.
    If the transformation matrix $A$ is diagonal with rational entries (or rational multiples
    of the same algebraic number) then the reachability problem is decidable.
    If the transformation matrix $A$ only has real eigenvalues, the reachability problem is
    conditionally decidable.
    The time-bounded reachability problem is conditionally decidable and unconditionally
    decidable in two dimensions.
    Some of our results rely on the
    decidability of certain logical theories ---
    namely the theory of the reals with exponential ($\RRexp$) and with bounded sine ($\RRexpsin$)---
    which have been proven decidable
    conditional on Schanuel's Conjecture --- a unifying conjecture in transcendence theory.
    We also obtain a hardness result for a mild generalization of the problem where the target is
    a simple set (hypercube of dimension $n-1$ or hyperplane) instead of a point.
    In this case, we show that the problem is at least as hard as the \emph{Continuous Positivity problem}
    if the control set is a singleton, or the \emph{Nontangential Continuous Positivity problem}
    if the control set is $[-1,1]$.
\end{abstract}

\maketitle

\section{Introduction}

This paper is concerned with \emph{linear time-invariant (LTI)
  systems}.  LTI systems are one of the most basic and fundamental
models in control theory and have applications in circuit design,
signal processing, and image processing, among many other areas.  LTI
systems have both discrete-time and continuous-time variants; here we
are concerned solely with the continuous-time version.

A (continuous-time) LTI system in dimension $n$ is specified by a
transition matrix $A \in \mathbb{Q}^{n\times n}$, a control matrix
$B \in \mathbb{Q}^{m\times n}$ and a set of controls
$U\subseteq \mathbb{R}^m$.  The evolution of the system is described
by the differential equation $x'(t)=Ax(t)+Bu(t)$ where $u:\RR\to U$
is a measurable function. Here we think of $u$ as an input (or control)
applied to the system. Note that the number of inputs is independent of the dimension:
it is possible to have only one input ($m=1$) in dimension $n$, or many inputs in
small dimension ($m>n$).

Given such an LTI system, we say that state $x_0\in\RR^n$ can \emph{reach}
state $y\in\RR^n$ if there exists $T \geqslant 0$ and a control\footnote{From
now on, all controls are necessarily measurable functions, we omit it
most of the time.} $u:[0,T]\to U$ such that the unique solution to the differential
system $x'(t)=Ax(t)+Bu(t)$ for $t\in(0,T)$ with initial condition $x(0)=x_0$
satisfies $x(T)=y$. Similarly, given $ t \geqslant 0$, We say that state $x_0$ can reach
state $y$ \emph{in time at most $ t$} if it can reach $y$ at a time $T \leqslant t$.
The problem of computing the set of all states reachable
from a given state has been an active topic of research for several decades.
Almost exclusively, the emphasis is typically
on efficient and scalable methods to over- and under-approximate the
reachable set~\cite{CattaruzzaASK15,GirardGM06,GirardG08,Kaynama2010OverapproximatingTR,SummersWS92}.
Furthermore, this problem has numerous practical applications and is a fundamental basic
block for the analysis of more complicated models, such as hybrid systems~\cite{dang00,Alur11,TMBO03}.
By contrast, there are relatively few results concerning the
\emph{decidability} of the reachable set---the focus of the present
paper.

We consider the \emph{LTI Reachability problem}:
given an LTI system, and target state $y$,
decide whether $\boldsymbol{0}$ can reach $y$. Specifically, we primarily focus on the case where the inputs
are \emph{saturated}, that is $U=[-1,1]^m$. Equivalently, one can think of $BU$ as being a zonotope~\cite{McMullen71}.
This naturally leads us to study the impact of
the control matrix ($B$) on the LTI Reachability problem. We will also consider the \emph{Bounded
Time LTI Reachability problem} where we are given an upper bound on the time allowed to reach
the target: given an LTI system, and target state $y$ and a time bound $T$,
decide whether $\boldsymbol{0}$ can reach $y$ in time at most $T$. Finally, we consider the
\emph{LTI Set Reachability problem} where the target $y$ becomes a set and we ask whether there exists
a reachable point within this set.

Close variants of the LTI Reachability problem include the Controllability problem
(set of points that can reach $\boldsymbol{0}$). It is also possible to consider the set of
points reachable from a given source $x_0$. Both problems are equivalent to the Reachability
problem either in backward time and/or with a modification of the control matrix and set.

The decidability of point-to-point reachability for linear systems is open although
for many different extensions and generalizations of the basic LTI model point-to-point reachability
has been shown undecidable (see discussion of related work).  While there are a number of classical
results on decidability of reachability for LTI systems in the literature,
these talk about ``universal'' reachability properties with almost the same names: \emph{null reachability}
(can one reach all states from the origin?)
and \emph{null controllability} (can one reach the origin from all
states?)~\cite{BlondelT99a}.  However these ``universal'' reachability
problems are very different from the point-to-point version
that we study. In particular, both null reachability and null
controllability are decidable in polynomial time using linear algebra.

One of the first results about (continuous-time) LTI Reachability problems is that
of Kalman~\cite{Kalman63} where the control sets are linear subspaces of $\RR^n$.
An important particular case is that of the \emph{Orbit problem}: given a matrix $A$,
an initial state $x_0$ and a state $y$, decide whether $y$ is in the \emph{orbit of $x$ under $A$},
\emph{i.e.} whether $y=e^{At}x_0$ for some\footnote{Although this problem is known as the ``Orbit'' problem, it really is a semi-orbit problem.
If one considers $t\in\RR$, the problem reduces to two semi-orbit problems.}
$t \geqslant 0$. This corresponds to the case
when the control is a singleton (or equivalently with non-zero $x(0)$ and zero control set).
This problem was shown to be decidable in polynomial time~\cite{Hainry08,CYH15}.
These results yield (polynomial-time) decidability when
the control sets are affine subspaces of $\RR^n$. An exact description of the
null controllable regions for general linear systems with saturating actuators was obtained \cite{HuLQ02},
however, this formula does not immediately yield an algorithm (see Section~\ref{sec:challenges}).

In this paper, we study the decidability of several special instances of the LTI Reachability problem with saturated inputs.
Specifically, we show the following results, all conditional on Schanuel's Conjecture:
\begin{itemize}
    \item In two dimensions ($n=2$), the reachability problem is decidable.
    \item If $A$ has real spectrum then the reachability problem is decidable.
    \item The time-bounded reachability problem is decidable.
\end{itemize}
These results are conditional in that they rely on the decidability of certain
mathematical theories, namely the theory of the reals with exponential ($\RRexp$) and with bounded sine ($\RRexpsin$).
Both theories are known to be decidable assuming Schanuel’s conjecture \cite{Macintyre1996}, a
major conjecture in transcendental number theory that is widely believed to be true.
We also manage to find some class of LTI with \emph{unconditionally} decidable reachability problem:
\begin{itemize}
    \item In two dimensions, when the $A$ has real spectrum
        and there is only one input ($m=1$).
    \item When $A$ is diagonalizable with rational eigenvalues (or rational multiples
        of the same algebraic number).
    \item When $A$ is real diagonal, there is only one input and it has at most two nonzero entries.
    \item When $A$ only has one eigenvalue, which is real, and there is only one input.
\end{itemize}
While those subclasses look ad-hoc, they all correspond to specific forms of the boundary of the reachable
set and the study of the transcendental points on this boundary. In particular, some of those cases
require some nontrivial theorems in transcendental number theory (Gelfond–Schneider, Lindemann-Weierstrass).
See \Cref{sec:challenges} for more details.

We also obtain a hardness result for a mild generalization of the problem where the target is a
simple set (either a hyperplane or compact convex set of dimension n-1) instead of a point, and the control set
is either $U=\set{u}$ or $U=[-1,1]$. In this case, we show that the problem is at least as hard as the
\emph{Continuous Skolem problem}
or the \emph{Nontangential Continuous Skolem problem} which asks whether the first component $x_1(t)$
of the solution to a linear differential equation
$x'(t)=Ax(t)$ has a zero (resp. nontangential zero). Showing decidability of any of these problems
would entail a major new effectiveness result in Diophantine approximation,
which suggests that the problem is very challenging.

\textbf{Related Work.}  It is well-known that besides linear
systems, most control problems are undecidable~\cite{BlondelT99,BlondelT00}.
For example, point-to-point reachability is hard or undecidable for \emph{piecewise linear
systems}~\cite{AsarinMP95,BlondelT99a,KoiranCG94,BournezKP18,KurganskyyPC08,Bell16,KPC07,Asarin2021}, for
\emph{saturated linear systems}~\cite{SiegelmannS95} and point-to-set reachability
is undecidable for polynomial systems, a consequence of \cite{GBC08}.
However, to the best of our knowledge, there are no (un)decidability results
within the class of LTI systems, except when the control sets are affine subspaces.
An exact description of the null controllable regions for general linear systems
with saturating actuators was obtained \cite{HuLQ02}, but it does not immediately
translate into a decidability result. On the other hand, the reachability problem
is well-known to be challenging in practice, even for LTI systems. There is a
vast literature on efficient and scalable methods to over- and under-approximate the
reachable set~\cite{CattaruzzaASK15,GirardGM06,GirardG08,Kaynama2010OverapproximatingTR,SummersWS92}.
However those methods, by construction, cannot lead to any decidability results
on their own. In fact, one can observe that a corollary of these methods is
that the ``only'' difficult part of the problem, in terms of decidability,
is the boundary of the reachable set.

A range of different control problems for discrete- and
continuous-time LTI systems under constraints on the set of controls
have been studied in the
literature~\cite{FijalkowOPP019,Cook80,Sontag84,SummersWS92,HuQiu98,HuMQ02,HuLQ02,TilS86,Grantham1975,
SchmitendorfB80,HeemelsC08,GirardG08,Jamak00,Zhao17}.
Kalman showed that when the control set is $U = \RR^m$ , the system is globally null-controllable
(every point can be controlled to 0) if and only if $(A,B)$ is controllable (see \Cref{def:controllable})
\cite{Kalman63}. Lee and Markus considered $U$ such that
$\mathbf{0} \in U \subset \RR^m$ and showed that if $(A,B)$ is controllable and all eigenvalues
have negative real parts, then the system is globally null-controllable \cite{LM}.
Sontag considered the problem of asymptotic null-controllability which asks if there is a control
that reaches the origin in the limit \cite{Sontag84}.
Summers discussed over estimation of the reachable set (from origin) by n-dimensional
ellipsoids when $U = [-1,1]^m$ \cite{SummersWS92}. Schmitendorf considered time varying matrices
$A(t)$ and $B(t)$ and gives a characterisation for a given point to be controllable when $U$
is compact \cite{SchmitendorfB80}, however this does not immediately yield an algorithm
(see Section~\ref{sec:challenges}). Lafferriere considered a different reachability problem
where the inputs are expressible in the first-order theory of the reals with some unknown
coefficients \cite{Lafferriere99,Lafferriere01}, and it was generalized by \cite{GCDXZ15,GanCLXZ16}
but our problem is of a very different nature because we do not require the input to have a closed-form.

\section{Examples}

The idea of using an external input to manipulate the state of some system to achieve a
certain goal is fundamental and everywhere in our lives. In order to give a better idea
of the problem we are trying to solve, we will informally explain the theory and its goals via some examples.

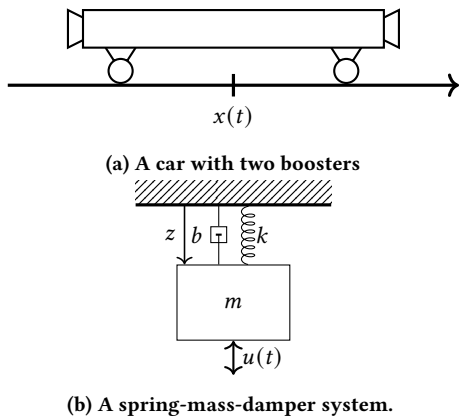
\begin{figure}[h]
    \begin{subfigure}[B]{0.4\textwidth}
        \centering
        \begin{tikzpicture}[scale=1]
            \draw[black,very thick,->] (2,0) -- (8,0); 
            \draw[black,very thick] (5,0.15) -- +(0,-0.3) node[below] (pos) {$x(t)$};
            \draw[black,thick] let \p1=(pos.center) in
                (\x1-2cm,0.5) node (bl) {} rectangle (\x1+2cm,1) node (tr) {};
            \draw[black,thick] ($(bl.center)+(0,0.1)$) -- ++(-0.2,-0.1) -- ++(0,0.5) -- ++(0.2,-0.1);
            \draw[black,thick] ($(tr.center)-(0,0.1)$) -- ++(0.2,+0.1) -- ++(0,-0.5) -- ++(-0.2,0.1);
            \draw[black,thick] let \p1=(bl.center) in
                (\x1+0.5cm,0.19) node (wheel left) {} circle[radius=0.16]
                (\x1+0.3cm,0.5) node (wlanchorl) {} -- ($(wlanchorl)!0.55!(wheel left.center)$)
                (\x1+0.7cm,0.5) node (wlanchorr) {} -- ($(wlanchorr)!0.55!(wheel left.center)$);
            \draw[black,thick] let \p1=(tr.center) in
                (\x1-0.5cm,0.19) node (wheel right) {} circle[radius=0.16]
                (\x1-0.3cm,0.5) node (wranchorl) {} -- ($(wranchorl)!0.55!(wheel right.center)$)
                (\x1-0.7cm,0.5) node (wranchorr) {} -- ($(wranchorr)!0.55!(wheel right.center)$);
        \end{tikzpicture}
        \caption{A car with two boosters\label{fig:car}}
    \end{subfigure}
    \begin{subfigure}[B]{0.4\textwidth}
        \centering
        \begin{tikzpicture}[scale=1.3]
            \fill[pattern=north east lines] (-1,0) rectangle (1, .25);
            \draw[black,very thick] (-1,0) -- (1,0);
            \node (mass) at (0,-1) [draw,minimum width=1.5cm,minimum height=1cm] {$m$};
            \draw[decorate,decoration={coil,segment length=3}]
                let \p1=(mass.north) in (\x1+0.15cm,\y1) -- ++(0,-\y1)
                node[midway,right] {$k$};
            \draw let \p1=(mass.north) in (\x1-0.15cm,\y1) -- ++(0,-\y1/2)
                node (damper) [draw,minimum width=0.1cm,minimum height=0.1cm] {}
                node [left,xshift=-0.1cm] {$b$};
            \draw[thick] let \p1=(mass.north) in (\x1-0.18cm,\y1/2) -- ++(0.06cm,0);
            \draw let \p1=(damper.north) in (\x1,\y1) -- ++(0,-\y1);
            \draw[thick,<->] (mass.south) -- ++(0,-.35) node[midway,right] {$u(t)$};
            \draw[semithick,->] let \p1=(mass.north) in (-0.5,0) -- ++(0,\y1)
                node[left,midway] {$z$};
        \end{tikzpicture}
        \caption{A spring-mass-damper system.
            \label{fig:mass_spring_damper}}
    \end{subfigure}
    \caption{Two examples of dynamical systems with controls.}
\end{figure}

Consider the toy example in \Cref{fig:car} (taken from Chapter 1 of \cite{MackiStrauss82}).
A car has two boosters, one at the front and one at the back. At time $t=0$,
it starts at some position $x_0\in\RR$ on the real line, with velocity $v_0$.
The objective is to reach the origin and stay there indefinitely, that is to reach the origin with a speed of $0$.
The external input in the above problem is the effects of the boosters that affect the acceleration directly,
thereby affecting the velocity and the position. We model the state of the system by its
position and velocity $S(t)=(x(t),v(t))\in\RR^2$. Assuming the front and rear boosters are similar and give a max acceleration of $M$ units,
we can model the dynamics  by
\[
    x'(t) = v(t),\qquad v'(t) = u(t)
\]
where $u(t)\in[-M,M]$: the acceleration is positive when the rear booster is on, and negative when
the front booster is on. We call $U=[-M,M]$ the control set. Combining both equations and writing it in matrix form gives us
\[
    S(0)=\begin{bmatrix}x_0\\v_0\end{bmatrix},\qquad 
    S'(t) = \begin{bmatrix} 0 & 1 \\ 0 & 0 \end{bmatrix} S(t) + \begin{bmatrix}0 & 0\\ 0 & 1 \end{bmatrix}u(t) = AS(t)+Bu(t).
\]
Here, the problem is to find or ``synthesize'' a control $u$ such that we reach $(0,0)$ from the initial point.
Note that in real life, we cannot change the control (booster output) arbitrarily fast, \ie{} not all functions $u$
are to be considered. In this work, we neglect this aspect and allow any function $u:\RR\to U$ that is \emph{measurable},
which is essentially the minimum mathematical condition for the problem to make sense. Observe that
already in this toy example, it is natural to consider a bound on the acceleration:
the control set is therefore bounded.

In the example of the car, we viewed the input as something under our control that we used to achieve some
objective. A dual view is to consider certain safety problems and check whether the input, now controlled by an adversary,
can be used to steer the system to a bad state. Consider the system in \cref{fig:mass_spring_damper},
a spring-mass-damper system with an external force acting on it. This typically models a vehicle's suspension.
For example, consider a bike travelling on a road and encountering a speed breaker or hump.
We are interested in the vertical movement of the tires, after they cross the hump. Here, the tire acts as the mass,
the damper and spring form the bike suspension, which provides shock absorption and the recoil force upon hitting the ground,
the road is modelled by the external force $u$.
A bad state is one when the tire's vertical movement is higher than certain admissible value which we want to avoid (for it
could damage or even break the suspension). Here the problem is to decide whether it is possible via some external input to
reach a bad state. Similar to the previous example, we could model the above system by an LTI system with a bounded
control set.

\section{Challenges}\label{sec:challenges}

A major challenge in solving the continuous-time reachability problem is the fact that there is
no simple formula, or more exactly, no formula that is immediately computable.
Given a LTI system $x'(t)=Ax(t)+Bu(t)$, there is a general expression for
$x(t)$ given $u$ that involves an integral and exponential of matrix (see Section~\ref{sec:math}):
\begin{equation}\label{eq:sol_lti}
    x(t)=e^{At}x_0+\int_0^te^{A(t-s)}Bu(s)\dd s.
\end{equation}
Therefore, the reachability problem is equivalent to checking whether there exists a $u$
such that \eqref{eq:sol_lti} is equal to the target $y$. Unfortunately, it seems impossible to
obtain a more useful closed-form formula without knowing more about the shape of $u$.
This is why we now assume that $x(0)=\mathbf{0}$ and $u:\RR\to U$ for some convex set $U$. The former
is without loss of generality\footnote{One can always shift the control set $U$ to ensure that $x(0)=\mathbf{0}$.}
and the latter is very common in the literature.

The simplest case is when $U=\RR^m$: the reachable set can be shown to be a linear subspace,
the image of the so-called controllability matrix $\begin{bmatrix}B&AB&\cdots A^{n-1}B\end{bmatrix}$
and therefore the reachability problem reduces to an orbit problem (if the image of the controllability
matrix is not the whole space, what happens on the remaining space is exactly an orbit problem).

A more interesting case, and the subject of this paper, is when $U$ is a compact convex polytope
and in particular a hypercube: $U=[-1,1]^m$. This is known as the \emph{saturated input}
case. It is not hard to see that when $x(0)=\mathbf{0}$ and $U$ is convex, the reachable set $\Reach$
is strictly convex. Unfortunately, the set $\Reach$ can still be very complicated and checking
whether a single point lies inside it turns out be a challenging problem. We are aware of two
distinct but similar results in this direction. Schmitendorf et al. \cite{SchmitendorfB80}
have some general conditions under which a control can steer a point to the origin. In the
particular case at hand, they turn out to be equivalent to another formulation by \cite{HuLQ02} where
the boundary of the reachable region is described by the sets of the form
\[
    \int_0^{\infty} e^{A t}b\sgn(c^Te^{A t}b)\dd t
\]
where $c \in \RR^n\setminus\set{\mathbf{0}}$ is a parameter and $b$ is a column of $B$. The main challenge is that evaluating this
integral is potentially a hard problem. In particular, if $A$ has a complex eigenvalue whose
argument is not rational multiple of $\pi$, then the sign of $c^Te^{A t}b$ will follow a completely
irregular pattern. In fact, the a priori simpler problem of deciding whether $c^Te^{A t}b$
will change sign \emph{at all} is exactly the continuous Skolem problem. This problem is open and has
been shown to be related to difficult number theoretical questions (see \Cref{sec:skolem}). Note, however,
that computing this integral, or rather deciding if this integral is less than some prescribed number,
does not necessarily require solving the continuous Skolem problem. In fact, a solution to Skolem
would not help per se (there could be infinitely many changes, whose values are not even algebraic),
and conversely, computing this integral does not necessarily help deciding the existence of a sign
change.

\bigskip
The approach we follow in this paper is to study the membership in the boundary.
Indeed, it is not hard to see (see \Cref{prop:complement_border_reach_lti_approx}) that if we can
decide membership into the boundary, then we can decide reachability.
Consider the example illustrated on \Cref{fig:ex_reach2d}: already in dimension $2$,
the boundary of the reachable set from $x(0)=0$ when $U=[-1,1]$, $B=\begin{bmatrix}b_1&b_2\end{bmatrix}$ and $A$ is stable is not trivial;
it consists of two smooth curves joining at singular points. In this case,
the boundaries are exactly (see \apxref{\Cref{sec:details_fig_reach}}{full version} for the details) the sets
\begin{align*}
    \partial\Reach(A,b_1)&=\set*{\pm\begin{bmatrix}2-4t^3\\3-6t^2\end{bmatrix}:t\in[0,1]},\\
    \partial\Reach(A,B)&=\set*{\pm\begin{bmatrix}4-4|t|^3\\6t^2\sgn(t)\end{bmatrix}:t\in[-1,1]}.
\end{align*}
Hence deciding membership of a point $(x,y)$ is exactly deciding
\begin{equation}\label{eq:equiv_formula_second_example}
    \exists t\in[-1,1].\, 4-4|t|^3=\pm x\wedge 6t^2\sgn(t)=\pm y
\end{equation}
which can be shown, after a bit of rewriting, to be in $\RRtarski$, the first-order theory of the reals.
Therefore, membership is decidable in this case by Tarski's Theorem. This example turns out to be easy
because $A$ is diagonal and all eigenvalues of $A$ are rational and therefore, the boundary can be described by polynomial
equations (see \Cref{prop:decide_real_mat_uncond}).

A small modification of this example already turns out to be much more challenging, if we consider the case where
\[
    A=\begin{bmatrix}-1&0\\0&-\sqrt{2}\end{bmatrix},
    \qquad
    B=\begin{bmatrix}1&1\\-1&1\end{bmatrix}.
\]
Then a similarly computation (see \apxref{\Cref{sec:details_ex2_reach}}{the full version}) shows that
\[
    \partial\Reach(A,B)=\set*{
        \pm\begin{bmatrix}-2|t|\\
            \sqrt{2}\left(1-|t|^{\sqrt{2}}\right)\sgn(t)
            \end{bmatrix}:t\in[-1,1]}.
\]
Hence deciding membership of a point $(x,y)$ is exactly deciding
\[
    \exists t\in[-1,1].\, -2|t|=\pm x\wedge \sqrt{2}\left(1-|t|^{\sqrt{2}}\right)\sgn(t)=\pm y
\]
which can be shown, after a bit of rewriting, to be in $\RRexp$, the first-order theory of the reals
with exponential. However, this formula is not in $\RRtarski$ because we need to raise $t$ to some irrational
power ($\sqrt{2}$) which leaves its decidability open a priori.
This example is a particular case of \Cref{prop:decide_real_eigen_rexp}:
when the eigenvalues of the matrix are real, the formulas
will only involve real exponentials and the boundary can be expressed in the theory of reals with
exponential ($\RRexp$). This theory is known to be decidable subject to Schanuel’s conjecture (see
\Cref{sec:transcendence}). Unfortunately, the formulas may further involve sine and cosine when $A$
has complex eigenvalues, and the theory of reals with sine and cosine is undecidable for example
by Richardson's theorem~\cite{Richardson68} and its improvements~\cite{Laczkovich03}
. This suggests
that the reachability problem is hard, but surprisingly, we have only been able to show some hardness
results in the case of set reachability.

It turns out that the second example is decidable by a similar argument to the proof of
\Cref{prop:decide_real_mat_uncond} using
\Cref{th:gelfond_schneider} (Gelfond-Schneider). Indeed, we can remove the absolute value
and $\sgn$ in \eqref{eq:equiv_formula_second_example} by doing a case distinction on the sign of $t$.
Then the first equation in \eqref{eq:equiv_formula_second_example} implies that $2t=x$ and hence the second that
$\sqrt{2}(1-(x/2)^{\sqrt{2}})=y$. But since $x$ is assumed to be algebraic, $(x/2)^{\sqrt{2}}$
must be transcendental hence $y$ cannot be algebraic, a contradiction. It follows that the boundary
contains no algebraic points which shows decidability of the problem. This reasoning, however, does
not seem to apply in the general case of a system of dimension $n=2$ with $m=2$ controls.
\medskip

The subclasses of linear systems that we identified for our decidability results,
while ad-hoc at first sight, really correspond to subclasses of exponential polynomial equations
that we can solve systematically.

\begin{figure*}
    \begin{center}
    \begin{tikzpicture}[
            fillcolor/.style={
                pattern color=darkgreen, pattern=crosshatch dots
            },
            edgecolor/.style={blue,line width=\fracszedge},
            ptcolor/.style={opacity=1},
            scale=0.75,
            ]
        \newcommand{\fracscale}{1}
        \newcommand{\fracszpt}{0.7pt}
        \newcommand{\fracszedge}{0.7pt}
        \newcommand{\fracsznorm}{0.7pt}
        \newcommand{\fracnormlen}{1}
        \newcommand{\fraccoord}{0pt}
        \begin{scope}[shift={(0,0)}]
            \draw[smooth,samples=100,domain=1:20,variable=\t,edgecolor,fillcolor]
                plot ({2-4/(\t^3)},{3-6/(\t^2)}) -- plot ({-2+4/(\t^3)},{-3+6/(\t^2)}) -- cycle;
        \end{scope}
        
        \begin{scope}[shift={(6,0)}]
            \draw[smooth,samples=100,domain=1:20,variable=\t,edgecolor,fillcolor]
                plot ({2-4/(\t^3)},{-3+6/(\t^2)}) -- plot ({-2+4/(\t^3)},{3-6/(\t^2)}) -- cycle;
        \end{scope}
        \begin{scope}[shift={(13,0)},scale=0.6]
            \draw[smooth,samples=100,domain=1:20,variable=\t,edgecolor,fillcolor]
                plot ({4-4/(\t^3)},{6/(\t^2)}) -- plot ({4-4/((21-\t)^3)},{-6/((21-\t)^2)}) --
                plot ({-4+4/(\t^3)},{-6/(\t^2)}) -- plot ({-4+4/((21-\t)^3)},{6/((21-\t)^2)}) --
                cycle;
        \end{scope}
        \draw (3,0) node {\scalebox{2}{$\boldsymbol{+}$}};
        \draw (9,0) node {\scalebox{2}{$\boldsymbol{=}$}};
        \begin{scope}[shift={(0,-4.5)}]
            \draw (3,0) node {$A=\begin{bmatrix}-\frac{1}{2}&0\\0&-\tfrac{1}{3}\end{bmatrix}$};
            \draw (-1,0) node {$b_1=\begin{bmatrix}1\\1\end{bmatrix}$};
            \draw (7,0) node {$b_2=\begin{bmatrix}-1\\1\end{bmatrix}$};
            \draw (12.5,0) node {$B=\begin{bmatrix}1&-1\\1&1\end{bmatrix}$};
        \end{scope}
    \end{tikzpicture}
    \end{center}
    \caption{Example of a simple LTI system: the reachable set is depicted in three cases.
        The two pictures on the left correspond to the case of one control ($B=b_1$ and $b_2$ respectively).
        The picture on the right corresponds to the case where $B=[b_1,b_2]$:
        the reachable set is then the Minkowski sum of the two reachable sets.
        \label{fig:ex_reach2d}}
\end{figure*}
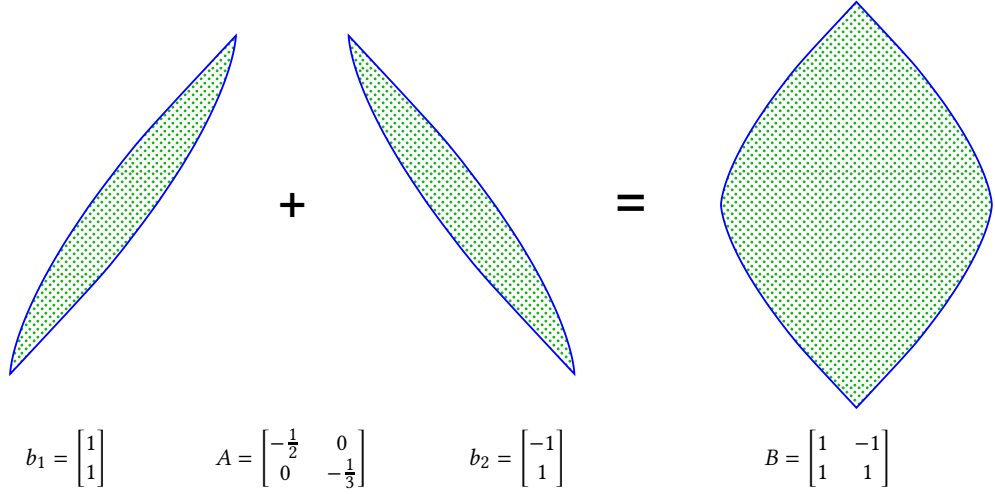

\section{Preliminaries}\label{sec:math}

We denote the usual Euclidean norm of vectors $x\in\CC^n$ by $\twonorm{x}$ and $\twonorm{A}$
the induced norm on matrices $A\in\CC^{n\times n}$. Any induced norm is consistent
($\twonorm{Ax} \leqslant \twonorm{A} \twonorm{x}$) and therefore submultiplicative
($\twonorm{AB} \leqslant \twonorm{A} \twonorm{B}$). Given a matrix $A\in\CC^{n\times n}$,
$e^{A}$ denotes the matrix exponential of $A$.
In particular, we have that $\twonorm{e^{A}} \leqslant e^{\twonorm{A}}$. We denote the boundary
of set $S$ by $\partial S$ and its closure by $\overline{S}$.
We denote the transpose of a vector or matrix $A$ by $A^T$ and its spectrum by $\Spectrum(A)$.
Additional preliminaries are in \apxref{\Cref{apx:prelim}}{available in the full version}.

\subsection{Control Theory}\label{sec:control_theory}

The following definition is standard in the literature of control theory.

\begin{definition}[Controllable]\label{def:controllable}
    A pair of matrices $(A,B)$ is called controllable if the rank of $[B,AB,\ldots,A^{n-1}B]$
    is $n$, where $A$ is an $n \times n$ matrix.
\end{definition}


Consider two vectors $c$ and $b$ in $\RR^n$ and define a function $f_{c,A,b}(t) = c^Te^{At}b$.
Then it follows from the Jordan decomposition that $f_{c,A,b}(t) = \sum_{j=1}^{m} P_j(t)e^{\theta_j t}$
where each $\theta_j$ is an eigenvalue of $A$
and $P_j$ a polynomial. The following properties of $f$ are well-known
(see e.g. \cite{hajek2009control}).

\begin{lemma}\label{lem:control_fc}
    Let $f_{c,A,b}(t)$ be the function defined above. Then
    \begin{itemize}
    \item if $f_{c,A,b} \neq 0$, then the number of zeros of $f_{c,A,b}$ in any bounded interval is finite,
    \item $(A,b)$ is controllable $\iff$ for all $c$, $f_{c,A,b} \neq 0$,
    \item if the eigenvalues of $A$ are real and $b,c$ are nonzero then $f_{c,A,b}$ has at most $n-1$
        zeros.
    \end{itemize}
\end{lemma}

Given a matrix $A$, a time bound $\tau\in\RR\cup\set{\infty}$ and a control
set $U$, define the null-controllable set $\Control$ and the reachable set $\Reach$ as
\begin{align*}
    \Control_\tau(A,B,U) &=
        \bigcup_{T=0}^\tau \set*{-\int_0^T\hspace{-2ex}e^{-A t}Bu( t)\dd t\,\Big|\,u:[0,T]\to U\text{ measurable}},
        \\
    \Reach_\tau(A,B,U) &=
        \bigcup_{T=0}^\tau \set*{\int_0^T e^{A t}Bu( t)\dd t\,\Big|\,u:[0,T]\to U\text{ measurable}}.
\end{align*}
When $\tau=\infty$, we simply write $\Control(A,B,U)$ and $\Reach(A,B,U)$.
It follows immediately from those definitions that $\Reach_\tau(A,B,U)=\Control_\tau(-A,B,-U)$ and $U=-U$ for a
hypercube (or any symmetric set). It is customary in the literature to express results about the
null-controllable sets. However, since we are interested in reachability questions, we find it more
convenient to state all results using the reachable set.

Define a matrix $A$ to be \emph{stable} if all its eigenvalues have negative real part, \emph{antistable}
if all its eigenvalues have positive real part, \emph{weakly-stable} (also called \emph{semi-stable} in \cite{HuLQ02})
if all its eigenvalues have nonpositive real parts and \emph{weakly-antistable}
if all its eigenvalues have nonnegative real parts.
Clearly $A$ is stable (resp. weakly-stable) if and only if $-A$ is antistable (resp. weakly-antistable).

In some cases, it is well-known that it is possible to decompose the system into its stable and weakly-antistable parts
(or dually into its antistable and semistable parts). In particular, this is possible when the control set is a hypercube:


\begin{proposition}[\cite{HuLQ02}]\label{prop:null_controllable_region_desc}
    Let $(A,B)$ be controllable, $U=[-1,1]^m$.
    \begin{itemize}
    \item If $A$ is weakly-antistable, then $\Reach = \RR^n$
    \item If $A$ is stable, then $\Reach$ is a bounded convex open set containing the origin
    \item If $A =  \begin{bmatrix}
                    A_1 & 0 \\
                    0 & A_2
                    \end{bmatrix}$
        where $A_1 \in \RR^{n_1\times n_1}$ stable and $A_2\in \RR^{n_2\times n_2}$ weakly-antistable
        and $B$ is partitioned as $\begin{bmatrix}B_1 \\ B_2\end{bmatrix}$ accordingly,
        then $\Reach = \Reach(A_1,B_1,U) \times \RR^{n_2}$.
    \end{itemize}
\end{proposition}

This fact, suggests that we need to study the reachable region of stable systems.
Decompose $B$ as $[b_1,\ldots,b_m]$ and assume that $U=[-1,1]^m$, then it is not hard to check that
\begin{equation}\label{eq:decompose_C_hypercube}
    \Reach(A,B,[-1,1]^m)=\Reach(A,b_1,[-1,1])+\cdots+\Reach(A,b_m,[-1,1])
\end{equation}
is a Minkowski sum of reachable regions in which $m=1$, \emph{i.e.} $B$ is a column vector.
In this case, one can obtain an explicit description of $\Reach$. Note however that
\eqref{eq:decompose_C_hypercube} does not, by itself, allows for a reduction to this simpler case:
even if we have an algorithm to decide membership in each smaller control set, deciding membership
in the Minkowski sum is nontrivial.


\begin{theorem}[\cite{HuLQ02}]\label{thm:bdry_descrptn}
    Let $A$ be stable, $b\in\RR^{n\times 1}$ such that $(A,b)$ is controllable
    and $U=[-1,1]$. Then $\Reach(A,b,U)$ is an open convex set containing $\boldsymbol{0}$ and its boundary
    is given by
    \[
        \partial\Reach(A,b,U) = \set*{\int_0^{\infty} e^{A t}b\sgn(c^Te^{A t}b)\dd t: c \in \RR^n\setminus \{\mathbf{0}\}}
    \]
    which is a strictly convex set.
\end{theorem}

\subsection{Continuous Skolem Problem}\label{sec:skolem}

The \emph{Continuous Skolem problem} is a fundamental decision problem concerning the reachability of
linear continuous-time dynamical system~\cite{BDJB10}. Given an initial point and a system of linear differential equation,
the problem asks whether the orbit ever intersects a given hyperplane. More precisely, given
a matrix $A\in\RR^{n\times n}$, two vectors $c,b\in\RR^n$, with rational (or algebraic coefficients),
the question is whether there exists $t \geqslant 0$ such that $c^Te^{At}b=0$. One can also consider
the bounded time version of this problem, where one asks about the existence of a zero at
time $t \leqslant T$ for some prescribed rational number $T$. While similar in spirit to
the Orbit problem (does the orbit reach a given point?), it is of a very different nature. In fact,
decidability of this problem is still open, even when restricting to the case of a bounded time
interval~\cite{ChonevOW15}. The Continuous Skolem problem admits several reformulations,
notably whether a linear differential equation or an exponential polynomial admits a zero~\cite{BDJB10}

The Continuous Skolem problem can be seen as the continuous analog of the \emph{Skolem problem},
which asks whether a linear recurrent sequence has a zero. The Skolem problem is a famously open
problem in number theory and computer science, which is known to be decidable up to dimension $4$
and  not known to be either decidable or undecidable starting from dimension $5$. We refer the reader
to \cite{OuaknineW15} for a recent survey on the Skolem problem.

The Continuous Skolem problem is only known to be decidable in very specific cases: in low
dimension, with a dominant real eigenvalue or a particular spectrum~\cite{BDJB10,ChonevOW15}.
Recent developments suggest that the Continuous Skolem problem is a very challenging problem.
Indeed, decidability of the problem in the case of two (or more) rationally linearly independent
frequencies would imply a new effectiveness result in Diophantine approximation that seem far off
at the moment~\cite{ChonevOW15}. Even decidability in the bounded case is nontrivial because of
\emph{tangential zeros}, and has only been shown recently subject to Schanuel’s Conjecture,
a unifying conjecture in transcendental number theory. While Schanuel’s Conjecture is widely
believed to be true, its far-reaching consequences suggest that any proof is a long way off.
For instance, it easily implies that $\pi+e$ is transcendental, but meanwhile the much weaker
fact that $\pi+e$ is irrational is still unknown! It should be noted however that from a complexity
theoretic perspective, the Continuous Skolem problem is only known to be at least NP-hard \cite{BDJB10}.

We now introduce the \emph{Continuous Nontangential Skolem problem}, a variant of this problem
where only zero-crossings are considered. Given a matrix $A\in\RR^{n\times n}$, two vectors $c,b\in\RR^n$,
with rational (or algebraic coefficients), the question is whether there exists $t \geqslant 0$ such that
$f(t)=0$ and $f'(t)\neq 0$, where $f(t)=c^Te^{At}b=0$. We call such a time $t$ a \emph{nontangential
zero}, as opposed to \emph{tangential zero} that would satisfy $f(t)=f'(t)=0$. Clearly any nontangential
zero is a zero but some systems admit tangential zeros.

We believe that this problem is essentially
as hard as the Continuous Skolem problem. Indeed, one of the reasons why the Continuous Skolem problem
is believed to be hard is a Diophantine hardness proof \cite{ChonevOW15}. In short, this reduction
shows that decidability would entail some major new effectiveness result in Diophantine approximation, 
namely computability of the Diophantine-approximation types of all real algebraic numbers.
But one can observe that the reduction of \cite{ChonevOW15} only
relies on nontangential zeros, hence decidability of the Nontangential Skolem problem would
also entail those Diophantine effectiveness results.

We note that there is a subtlety in the definition of the Nontangential Skolem problem: one needs to decide
the existence of nontangential zeros but it is entirely possible that it also has some
tangential zeros. Hence, even over a bounded interval, it is not clear that the problem is decidable.
For instance, a Newton-based method would not be able to distinguish between a tangential or a
nontangential zero using a finite number of iterations. The problem easily becomes decidable,
\emph{over a bounded interval} under the premise that there are no tangential zeros. We also believe
that a variant of the decidability argument in \cite{ChonevOW15} would show that the problem is decidable over
bounded interval, assuming Schanuel’s conjecture.
As we have seen before, the problem is hard for Diophantine-approximation types over unbounded intervals.

\subsection{First-order theory of the reals}

A \emph{sentence} in the first-order theory of the reals is (although one can allow more general
expression that interleave quantifiers and connectives) an expression of the form
$\phi=Q_1x_1\cdots Q_nx_n.\psi(x_1,\ldots,x_n)$ where each $Q_1,\ldots,Q_n$ is one of the quantifiers
$\exists$ or $\forall$, and $\psi$ is a Boolean combinations (built from connectives
$\wedge$, $\vee$ and $\neg$) of \emph{atomic predicates} of the form $P(x)\sim 0$ where $P$
is a polynomial with integer coefficients and $\sim$ is one of the relations $<,\leqslant,=,>,\geqslant,\neq$.
A theory is said to be decidable if there is an algorithm that, given a sentence, can determine if
it is true or false. A famous result by Tarksi is that first-order theory of reals admits
quantifier elimination and is decidable. In this paper, we denote by $\RRtarski$ this theory,
formally this is the first-order theory of the structure $(\RR,0,1,<,+,\cdot)$.

\begin{theorem}[Tarski's Theorem~\cite{Tarski}]\label{th:fo_reals}
    The first-order theory $\RRtarski$ of reals is decidable.
\end{theorem}

We note that although the theory only allows integer coefficients, one can easily introduce algebraic
coefficients by creating new variables and express that they are the roots on some polynomial.
See also \cite{Renegar,BPR06} for  more  efficient decision procedures for the first-order theory of reals.

\subsection{Transcendental number theory}\label{sec:transcendence}

A complex number is said to be \emph{algebraic} if it is a root of a nonzero polynomial with integer coefficients.
We denote by $\QQbar$ the field of all algebraic numbers. A non-algebraic number is called \emph{transcendental}.
We will use that all field operations on algebraic numbers (including comparisons) are effective,
see e.g. \cite{BPR06}.
We will use transcendence theory in our proofs, essentially to argue that some equalities between
two numbers are impossible. A classical results concerns powers of algebraic numbers.

\begin{theorem}[Gelfond–Schneider]\label{th:gelfond_schneider}
    If $a$ and $b$ are algebraic numbers with $a\neq 0,1$ and $b$ irrational, then any
    value\footnote{In general, $a^b$ is defined by $e^{b\log a}$ and can have several values
    depending on the branch of the logarithm.} of $a^b$ is transcendental.
\end{theorem}

An important generalization of this result is the Lindemann-Weierstrass Theorem. In particular,
we will use the following reformulation by Baker:

\begin{theorem}[Lindemann-Weierstrass, Baker's reformulation]\label{th:lindermann_weierstrass}
    If $\alpha_1,\ldots,\alpha_k$ are distinct algebraic numbers, then $e^{\alpha_1},\ldots,e^{\alpha_k}$
    are linearly independent over the algebraic numbers.
\end{theorem}

Our results in some cases depend on Schanuel’s conjecture, a unifying conjecture in
transcendental number theory \cite{Lang} that generalizes many of the classical results in
the field (including~\Cref{th:gelfond_schneider,th:lindermann_weierstrass}).
The conjecture states that if $\alpha_1,\ldots,\alpha_k \in \CC$ are linearly
independent over $\QQ$ then some $k$-element subset of
$\{\alpha_1,\ldots,\alpha_k,e^{\alpha_1},\ldots,e^{\alpha_k}\}$ is
algebraically independent.

Assuming Schanuel's Conjecture, MacIntyre and Wilkie~\cite{WilkieMacintyre} have
shown decidability of the first-order theory of the expansion of the
real field with the exponentiation function and the $\sin$ and $\cos$
functions restricted to bounded intervals.

\begin{theorem}[Wilkie and MacIntyre]\label{th:wilkie-macintyre}
    If Schanuel's conjecture is true, then, for each $n \in \NN$,
    the first-order theory of the structure
    $(\RR,0,1,<,+,\cdot,\exp,\cos\restriction_{[0, n]},\sin\restriction_{[0,n]})$
    is decidable.
\end{theorem}

In the rest of the paper, we denote by $\RRexp$ the first-order theory of the reals with the exponential,
and $\RRexpsin$ the first-order theory of the reals with the exponential and the $\sin$ and $\cos$
functions restricted to a bounded interval.

\section{Decidability}

The goal of this section is to study the decidability of the LTI Reachability problem in various special
cases. We will always restrict ourselves to the case where the control set is a hypercube $U=[-1,1]^m$.
Surprisingly, and despite the explicit description given by \Cref{thm:bdry_descrptn}, this problem
remains challenging (see \Cref{sec:hardness} for some hardness results). A well-known observation,
already made in \Cref{sec:control_theory} is that we can simplify the problem when the input lies
in a hypercube.

\begin{lemma}[\apxref{\Cref{apx:reduce_controllable_minkowski}}{See full version}]\label{prop:reduce_controllable_minkowski}
    For any $A\in\RR^{n\times n}$, $\tau\in\RR\cup\set{\infty}$ and $B\in\RR^{n\times m}$,
    if $A$ is stable or $\tau<\infty$ then there exists
    computable real matrices\footnote{Note that the $C_i$ can be of lower dimension that $n$ and the $P_i$
    are not necessarily square.} $C_1,\ldots,C_m,P_1,\ldots,P_m$ such that
    \[
        \Reach_\tau(A,B,[-1,1]^m)=\sum_{i=1}^mP_i\Reach_\tau(C_i,b_i,[-1,1])
    \]
    where the $b_i$ are the columns of $B$, $\Spectrum(C_i)\subseteq\Spectrum(A)$ and
    $(C_i,b_i)$ is controllable for all $i$. In particular,
    if for every $i$, the membership in $\Reach_\tau(C_i,b_i,[-1,1])$ or in $\partial \Reach_\tau(C_i,b_i,[-1,1])$
    is expressible in a theory $\RRtheory$
    that contains $\RRtarski$, then membership in $\Reach_\tau(A,B,[-1,1]^m)$ is expressible in that theory
    $\RRtheory$.
\end{lemma}

A second observation is that, for the purpose of decidability, we can focus on the boundary of the reachable set.
Indeed, we can compute arbitrarily good approximations of the boundary and hence, solve the problem if we
know that the target is not on the boundary. This only gives a semi-decision for the problem however
because the algorithm will never conclude when the target is exactly on the boundary.
If we can decide if an algebraic point is on the boundary, then we can either immediately conclude
(target on the boundary) or make sure that the semi-decision procedure will finish (target not on the
boundary). The following result is well-known:

\begin{lemma}[\apxref{\Cref{apx:complement_border_reach_lti_approx}}{See full version}]\label{prop:complement_border_reach_lti_approx}
    There is an algorithm that, given\footnote{In fact, this is still true for algebraic and even computable coefficients.
    A real is computable if one can produce arbitrary precise rational approximations of it.}
    $A\in\QQ^{n\times n}$ stable and $B\in\QQ^{n\times k}$ and $p\in\NN$,
    computes two convex polytopes $P_-$ and $P_+$
    such that $P_-\subseteq \partial\Reach(A,b,[-1,1])\subseteq P_+$ and the Hausdorff distance\footnote{Recall
    that the Hausdorff distance, which measures how far two sets are from each other, between two sets $X$ and $Y$ is defined by
    $\operatorname{d}(X,Y)=\max\left(\sup_{x\in X}\inf_{y\in Y}\lVert{x-y}\rVert,\sup_{y\in Y}\inf_{x\in X}\lVert{x-y}\rVert\right)$.}
    between
    $P_-$ and $P_+$ is less than $2^{-p}$.
\end{lemma}

We start with the simplest case where $A$ is already diagonal. In fact, this seemingly easy case
is already difficult and we only manage to solve it unconditionally in some cases.

\begin{proposition}\label{prop:decide_real_mat_uncond}
    The LTI Reachability problem is decidable when $U=[-1,1]^m$
    and one of the following conditions holds:
    \begin{itemize}
        \item $A$ is real diagonal, $B$ is a column (\emph{i.e.} $m=1$) and it has at most 2 nonzero entries,
        \item $A$ is real diagonalizable and its eigenvalues are rational, or a rational multiple of the same algebraic number,
        \item $A$ only has one eigenvalue, which is real, and $B$ is a column (\emph{i.e.} $m=1$).
    \end{itemize}
\end{proposition}

\begin{proof}
    We start by observing that in the second case, we can reduce to the case where $A$ is diagonal. Indeed,
    since $A$ is real diagonalizable, we can write $A=P^{-1}MP$ where $P$ is real and $M$ is diagonal.
    Note that $M$ satisfies all the assumptions since it contains only the eigenvalues of $M$. Furthermore,
    the reachable set is easily observed to be $\Reach(A,B,U)=P^{-1}\Reach(M,PB,U)$
    hence it is equivalent to decide if $Py\in \Reach(M,PB,U)$.

    Assume that $A$ is diagonal (this covers the first two cases of the theorem with the above remark).
    Write $A=\diag(\lambda_1,\ldots,\lambda_n)$ with $\lambda_i \ge \lambda_{i+1}$ without loss
    of generality. Decompose $A$ into $A=\diag(A_1,A_2)$ where $A_1$ contains the nonnegative $\lambda_i$
    and $A_2$ the negative ones. Then $A_1$ is weakly antistable and $A_2$ is stable. Decompose $B$
    into $B_1$ and $B_2$ accordingly. Then by \Cref{prop:null_controllable_region_desc},
    $\Reach(A,B,U)=\RR^{n_1}\times \Reach(A_2,B_2,U)$. Then by \Cref{prop:reduce_controllable_minkowski},
    we have that
    $\Reach(A_2,B_2,[-1,1]^m)=\sum_{i=1}^mP_i\Reach(C_i,b_i,[-1,1])$
    where the $b_i$ are the columns of $B_2$ and $(C_i,b_i)$ is controllable for all $i$.

    We now assume that $A$ is diagonal with negative eigenvalues, $B=b$ is a column vector,
    $(A,b)$ is controllable and $U=[-1,1]$. Write $A=\diag(\mu_1,\ldots,\mu_k)$ where the $\mu_i$ are negative.
    Then by \Cref{thm:bdry_descrptn}, we have that
    $\partial\Reach(A,b,[-1,1]) = \set*{\beta_c: c \in \RR^n\setminus \{\mathbf{0}\}}$, where
    \[
        \beta_c:=\int_0^{\infty}e^{At}b\sgn(c^Te^{A t}b)\dd t
    \]
    And observe that
    \[
        f_c(t):=c^Te^{A t}b=c^T\diag(e^{\mu_1 t},\ldots,e^{\mu_k t})b=\sum_{i=1}^kc_ie^{\mu_i t}b_i.
    \]

    \textbf{If all entries of $A$ are a rational multiple of the same algebraic number:}
    write\footnote{We can put the least common denominator in $\alpha$, hence they become integer multiples.}
    $\mu_i=p_i\alpha$ where $p_i\in\ZZ$ and $\alpha\in\QQbar$. Then
    \[
        f_c(t)=\sum_{i=1}^kc_ib_i\left(e^{\alpha t}\right)^{p_i}=Q(c,e^{\alpha t})
    \]
    where $Q$ is a polynomial with algebraic coefficients. Let $d$ be the
    degree of $Q(c,\cdot)$ (that does not depend on $c$ but only on the $\mu_i$), then $Q(c,\cdot)$ has at most
    $d$ nontangential\footnote{We are only interested in zero-crossings, since tangential zeros
    do not change the integral. In doing so, we also get for free that all nontangential zeros
    have multiplicity 1, hence they are all distincts.}
    zeros, call them $z_1<z_2<\cdots<z_k$. Each gives rise to some unique $t_i$ satisfying $z_i=e^{\alpha t_i}$.
    It follows that, up to a sign,
    \begin{align*}
        \pm\beta_c
            &=\int_{0}^{t_1}e^{At}b\dd t-\int_{t_1}^{t_2}e^{At}b\dd t+\cdots+(-1)^{k}\int_{t_k}^{\infty}e^{At}b\dd t\\
            &=A^{-1}\big((e^{At_1}-I)-(e^{At_2}-e^{At_1})+\cdots-(-1)^ke^{At_k}\big)b\\
            &=A^{-1}\left(2\sum_{i=1}^k(-1)^{i-1}e^{At_i}-I\right)b.
    \end{align*}
    In particular, since $A$ is diagonal, the $j^{th}$ component of $\beta_c$ is
    \begin{align*}
        \beta_{c,j}
            &=\pm\frac{1}{\mu_j}\left(2\sum_{i=1}^k(-1)^{i-1}e^{-\mu_jt_i}-1\right)b_j\\
            &=\pm\frac{1}{\mu_j}\left(2\sum_{i=1}^k(-1)^{i-1}\left(e^{-\alpha t_i}-1\right)^{p_i}\right)b_j\\
            &=\pm\frac{1}{\mu_j}\left(2\sum_{i=1}^k(-1)^{i-1}z_i^{p_i}-1\right)b_j\\
            &=R_{k,j}(z_1,\ldots,z_k)
    \end{align*}
    where $R_{k,j}$ is a polynomial with algebraic coefficients that does not depend on $c$. Note
    that the sign can be determined easily: it is the sign of $f_c(0)=\sum_{i=1}^kb_ic_i$. We can now
    write a formula in the first-order theory of the reals to express that a target $y$ is on the
    border\footnote{We write $Q'$ for $\tfrac{\partial Q(c,z)}{\partial z}$ which is also a polynomial.}:\
    \begin{align*}
        \Phi(y)
            &:=\exists c.c\neq0\bigwedge\bigvee_{k=0}^d\Phi_k(y,c)
        \intertext{to check for a point on a border in direction $c$,}
        \Phi_k(y,c)
            &:=\exists z_1,\ldots,z_k.\Psi_k(c,z)\bigwedge\Psi_k'(c,z)\bigwedge
                \bigwedge_{j=1}^n\left(y_j=R_{k,j}(z)\right)
        \intertext{to match the target with some parameters $z_1,\ldots,z_k$,}
        \Psi_k(c,z)
            &:=\bigwedge_{i=1}^k \left(Q(c,z_i)=0\wedge Q'(c,z_i)\neq 0\right)
                \bigwedge 0<z_1<\cdots<z_k
        \intertext{to check that the $z_i$ are zeros of $Q(c,\cdot)$,}
        \Psi_k'(c,z)
            &:=\forall u.(u>0\wedge Q(c,u)=0\wedge Q'(c,u)\neq 0)\Rightarrow\bigvee_{i=1}^k u=z_i
        \intertext{to check the $z_i$ are the only zeros of $Q(c,\cdot)$.}
    \end{align*}
    We have shown that membership in the border is expressible is $\RRtarski$, which shows that
    membership in the entire reachable set is expressible in $\RRtarski$ by \Cref{prop:reduce_controllable_minkowski},
    and hence decidable.

    \textbf{If $b$ has at most two nonzero entries:} let $i\neq j$ be those two entries, then
    \begin{align*}
        f_c(t)=0
            &\,\Leftrightarrow\, c_ie^{\mu_i t}b_i+c_je^{\mu_j t}b_j=0\\
            &\,\Leftrightarrow\, 1+\tfrac{c_jb_j}{c_ib_i}e^{(\mu_i-\mu_j)t}=0\\
            &\,\Leftrightarrow\, t=t_1:=\tfrac{1}{\mu_i-\mu_j}\ln\tfrac{-c_ib_i}{c_jb_j}.
    \end{align*}
    We assume that we are not in the previous case, so in particular $\mu_i$ and $\mu_j$ must be
    distinct and $\QQ$-linearly independent. If $c_i=0$ then $f_c$ has no zero unless $c_j=0$,
    in which case it is constant equal to zero and $\beta_c=0$. When $c_i=0$ and $c_j\neq0$, the sign
    of $f_c$ is constant. If $c_ib_i=-c_jb_j$ then the only zero of $f_c$ is at $t=0$ and the sign is then constant
    once again. In all case where the sign is constant on $(0,\infty)$, we have that
    \[
        \beta_c=\pm\int_0^{\infty}e^{At}b\dd t=-A^{-1}b
    \]
    which is algebraic and hence can easily be checked against the target. In all other cases, we have
    \begin{align*}
        \beta_c
            &=\int_0^{t_1}e^{At}b\dd t-\int_{t_1}^\infty e^{At}b\dd t\\
            &=A^{-1}(2e^{At_1}-I_n)b=2A^{-1}e^{At_1}b-A^{-1}b.
    \end{align*}
    Recall that $b$ has two nonzero entries and $A$ is diagonal, hence $A^{-1}e^{At_1}b$ also
    has two nonzero entries:
    $\mu_i^{-1}e^{\mu_it_1}b_i$ and $\mu_j^{-1}e^{\mu_jt_1}b_j$ respectively. We now argue that
    one or both of those values are transcendental which prevents
    the target from being on the border. Observe that $\mu_i^{-1}e^{\mu_it_1}b_i$ is algebraic if and only if
    $e^{\mu_it_1}$ is algebraic. But
    \[
        e^{\mu_it_1}
            =e^{\tfrac{\mu_i}{\mu_i-\mu_j}\ln\tfrac{-c_ib_i}{c_jb_j}}
            =\left(\tfrac{-c_ib_i}{c_jb_j}\right)^{\tfrac{\mu_i}{\mu_i-\mu_j}}
    \]
    which is transcendental by \Cref{th:gelfond_schneider} if $\tfrac{\mu_i}{\mu_i-\mu_j}$ is irrational,
    since we assumed that $\tfrac{-c_ib_i}{c_jb_j}$ is not $0$ or $1$.
    Therefore, $\beta_c$ is algebraic only when
    $\tfrac{\mu_i}{\mu_i-\mu_j},\tfrac{\mu_j}{\mu_i-\mu_j}\in\QQ$. This would imply that $\tfrac{\mu_i}{\mu_j}\in\QQ$,
    a contradiction since we assume that they are $\QQ$-linearly independent. In summary, the only
    two possible algebraic points on the border are $0$ and $A^{-1}b$ and it is easy to check
    (i) if they are indeed on the border, (ii) if they are equal to the target. Clearly one can write
    a formula in $\RRtarski$ to decide if this is the case.

    At this point, we can conclude for the general case because we assume that $B$ only consist of one column.
    Note that we would not be able to conclude if $B$ had several columns because we can only write a formula for
    $\partial\Reach(A_i,b_i,[-1,]1)\cap\QQbar^n$. Indeed, if we have two convex sets $C$ and $D$,
    then $\partial(C+D)\subseteq\partial C+\partial D$ but in general we do not have $\partial(C+D)\cap\QQbar^n
    \subseteq(\partial C\cap\QQbar^n)+(\partial D\cap\QQbar^n)$. A simple counter-example is $C=[0,\pi]$
    and $D=[0,4-\pi]$: then $C+D=[0,4]$, $\partial(C+D)\cap\QQbar=\set{0,4}$ but $\partial C\cap\QQbar=\set{0}$
    and $\partial D\cap\QQbar=\set{0}$. It is unclear whether such a counter-example can be built with
    actual reachable sets however.

    \textbf{If $A$ only has one eigenvalue which is real:} by using \Cref{prop:reduce_controllable_minkowski}
    as before, it suffices to show that membership is expressible in $\RRtarski$ for some controllable
    pairs $(A_i,b_i)$. The crucial point here is that the spectrum of $A_i$ is included in that of $A$.
    Since $A$ has a unique eigenvalue $\lambda$, $A_i$ also has a unique eigenvalue $\lambda$.
    If $\lambda>0$, the reachable set is $\RR^{n_1}$ by \Cref{prop:null_controllable_region_desc}, hence expressible
    in $\RRtarski$.

    We now assume that $\lambda<0$, $B=b$ is a column vector,$(A,b)$ is controllable and $U=[-1,1]$.
    Then by \Cref{thm:bdry_descrptn}, we have that
    $\partial\Reach(A,b,[-1,1]) = \set*{\beta_c: c \in \RR^n\setminus \{\mathbf{0}\}}$, where
    \[
        \beta_c:=\int_0^{\infty}e^{At}b\sgn(c^Te^{A t}b)\dd t
    \]
    But since $A$ has a unique real eigenvalue $\lambda$, $e^{At}=e^{\lambda t}P(t)$ where $P(t)$
    is a matrix where each entry is a polynomial in $t$ with real algebraic coefficients. It follows
    that $c^Te^{At}b=e^{\lambda t}Q(c,t)$ where $Q$ is a polynomial with algebraic coefficients and
    $Q(c,\cdot)$ has at most $d$ nontangential zeros
    (tangential zeros do not change the integral) which are distinct, and where $d$ is independent of $c$.
    We can write a formula in $\RRtarski$ to express those zeros $t_1<t_2<\cdots<t_k$. Furthermore,
    note that by integration by part, we have
    \[
        \int_u^ve^{At}b\dd t=\int_u^ve^{\lambda t}P(t)\dd t=\big[e^{\lambda t}R(t)\big]_u^v
    \]
    where $R$ is some polynomial matrix with algebraic coefficients.
    It follows that, up to a sign,
    \begin{align*}
        \pm\beta_c
            &=\int_{0}^{t_1}e^{At}b\dd t-\int_{t_1}^{t_2}e^{At}b\dd t+\cdots+(-1)^{k}\int_{t_k}^{\infty}e^{At}b\dd t\\
            &=\big[e^{\lambda t}R(t)\big]_0^{t_1}b-\big[e^{\lambda t}R(t)\big]_{t_1}^{t_2}b+\cdots+
                (-1)^{k}\big[e^{\lambda t}R(t)\big]_{t_k}^\infty b\\
            &=\left(2\sum_{i=1}^k(-1)^{i-1}e^{\lambda t_i}R(t_i)-R(0)\right)b.
    \end{align*}
    In particular, the $j^{th}$ component of $\beta_c$ is of the form
    \begin{align*}
        \beta_{c,j}
            &=S_{j,0}(b)+\sum_{i=1}^ke^{\lambda t_i}S_{j,i}(t_1,\ldots,t_k,b)
    \end{align*}
    where the $S_i$ are polynomials with algebraic coefficients. But now recall that the $t_i$ are algebraic
    since they are the roots of $Q(c,\cdot)$ and they are distinct because they are in fact the nontangential
    zeros. Furthermore, the target $\beta$ has algebraic coordinates. Hence, by \Cref{th:lindermann_weierstrass}
    (take $\alpha_i=\lambda t_i\in\QQbar$ and add $\alpha_0=0$), the only way this can happen
    is if $S_{j,i}(t_1,\ldots,t_k,b)=0$ for $1\leqslant i\leqslant k$ and $\beta_{c,j}=S_{j,0}(b)$.
    Crucially, those conditions do not involve any exponentials so we can express all those conditions in $\RRtarski$.
    Here again, we can only conclude when $B$ is a column, because we only have a formula for the algebraic
    point on the border of each controllable system.
\end{proof}

The main obstacle to generalizing this result is that the Boolean formulas involved in description
of the boundary become too complicated, either involving three distinct exponentials or a combination
of exponentials and polynomials (exponential polynomial). In fact, deciding if an exponential polynomial has zero
is exactly the Continuous Skolem problem, and is not known to be decidable, even for real eigenvalues.
We can recover decidability if we assume that the first-order theory of the reals with exponential is decidable.
This is known to be true if Schanuel’s conjecture hold, see \Cref{th:wilkie-macintyre}.

\begin{proposition}[\apxref{\Cref{apx:decide_real_eigen_rexp}}{See full version}]\label{prop:decide_real_eigen_rexp}
    The LTI Reachability problem when $U=[-1,1]^m$ and $A$ has real eigenvalues reduces to deciding $\RRexp$.
    In particular, it is decidable if Schanuel’s conjecture is true.
\end{proposition}

One cannot easily generalize the previous result to any matrix because complex eigenvalues involve
expression with $\exp$ and $\sin$ over unbounded domains. It is well-known that the first-order
theory of reals with unbounded $\sin$ is undecidable (by embedding of Peano arithmetic). This explains
why very few results are known about the Continuous Skolem problem in the unbounded case, even
assuming Schanuel’s conjecture. Nevertheless, one can show that in dimension two, only bounded
sine and cosine are necessary to solve the problem.

\begin{proposition}[\apxref{\Cref{apx:decide_complex_eigen_2d_rexpsin}}{See full version}]\label{prop:decide_complex_eigen_2d_rexpsin}
    The LTI Reachability problem in dimension $n=2$ when $U=[-1,1]^m$ reduces to deciding $\RRexpsin$.
    In particular, it is decidable if Schanuel’s conjecture is true.
\end{proposition}

In fact, dimension $2$ is special enough that we can show unconditional decidability of the
reachability problem if $A$ has real eigenvalues and $B$ is a column (i.e. there is only one input).

\begin{proposition}[\apxref{\Cref{apx:decide_2d_real_eigen}}{See full version}]\label{prop:decide_2d_real_eigen}
    The LTI Reachability problem in dimension $n=2$ when
    $U=[-1,1]$, $B$ is a column and $A$ has real eigenvalues
    is decidable.
\end{proposition}

Finally, another way to avoid the use of unbounded sine and cosine is to consider the Bounded Time LTI
Reachability problem, which is also very natural in control theory.

\begin{proposition}[\apxref{\Cref{apx:decide_bounded_time_rexpsin}}{See full version}]\label{prop:decide_bounded_time_rexpsin}
    The Bounded Time LTI Reachability problem when $U=[-1,1]^m$ and the time bound is algebraic
    reduces to deciding $\RRexpsin$. In particular, it is decidable if Schanuel’s conjecture is true.
\end{proposition}

\section{Hardness}\label{sec:hardness}

We saw in the previous section that the LTI Reachability problem seems very challenging, requiring
powerful tools like the first-order theory of the reals with exponential and Schanuel’s conjecture.
In this section, we give some evidence that the problem is indeed difficult. Our first observation
is that, in some sense, the LTI Set Reachability problem trivially contains the Skolem problem
when the input set is $\set{0}$, in other words, when there is no input.

\begin{theorem}[\apxref{\Cref{apx:lti_hard_skolem}}{See full version}]\label{th:lti_hard_skolem}
    The Continuous Skolem problem reduces to the LTI Set Reachability problem
    with input set $U=\set{z}$, where $z=(1,\ldots,1)$, the matrix $A$ is stable
    and the target set is a compact convex set of dimension $n-1$.
\end{theorem}

However, we are not really satisfied with this hardness result. Indeed, the problem is fundamentally
different when $U$ is a singleton. To see that, observe that if $U=\set{z}$ for some $z\in\RR^n$,
then the reachable set is just the orbit of $z$ under $x'=Ax$. In particular, we can see that if $A$
is stable then the orbit is a \emph{closed set} minus\footnote{The orbit converges
to $0$ but never reaches it, hence the reachable set is the closure (closed set) minus $0$.} an algebraic point (0). In particular,
we can trivially decide whether an algebraic point is $\boldsymbol{0}$ or not, so deciding reachability is
really about \textbf{deciding membership in a closed set}.
Now compare that with the situation when $U=[-1,1]^n$: by \Cref{prop:null_controllable_region_desc},
when $A$ is stable, the reachable set is \textbf{open}. This topological difference can lead to some
difficulty because deciding membership in the boundary may involve some difficult transcendence results.
For this reason, it is important to study hardness when $U$ is not a singleton.

We show that the problem remains hard when $U=[-1,1]$, by reducing to the Continuous Nontangential Skolem problem.
Recall that this problem, asks whether an exponential polynomial
(or equivalently a linear differential equation) has a zero-crossing (nontangential zero). We argued
in \Cref{sec:skolem} that this problem is essentially as hard as the Skolem problem.

\begin{theorem}\label{th:lti_nontangen_hard_skolem}
    The Continuous Nontangential Skolem problem reduces to the LTI Set Reachability
    problem with a single saturated input, \emph{i.e.} $x'=Ax+bu$ with $A$ stable,
    $b\in\RR^{n}$ and $u(t)\in[-1,1]$,
    and the target set can be chosen to be either a hyperplane, or a convex compact set of dimension
    $n-1$.
\end{theorem}

\begin{proof}
    Let $c,A,b$ be an instance of the Continuous Nontangential Skolem problem. Let
    $f_c(t)=c^Te^{At}b$, the problem asks whether $f_c$ has any zero-crossing at $t\geqslant 0$.
    Note that for any $\alpha>0$, $f_c(t)$ is zero-crossing if and only if $e^{-\alpha t}f_c(t)$
    is zero-crossing. Furthermore, $e^{-\alpha t}f_c(t)$ is still an exponential polynomial, so
    without loss of generality we can assume that all eigenvalues of $A$ have negative real parts,
    by taking $\alpha$ sufficiently large. In other words, we can assume that $A$ is stable.
    In particular, $A$ must be invertible.

    We now show that we can assume that $(A,b)$ is controllable.
    Let $V=\Span[b,Ab,\ldots,A^{n-1}b]$ where $n$ is the dimension of $A$, and assume that $\dim V<n$.
    Then $b\in V$ and $AV\subseteq V$ by Cayley–Hamilton theorem, so by an orthogonal change of basis $P$,
    \[
        P^{-1}c=\begin{bmatrix}c_V\\ *\end{bmatrix},\qquad
        P^{-1}AP=\begin{bmatrix}A_V&*\\0&*\end{bmatrix},\qquad
        P^{-1}b=\begin{bmatrix}b_V\\0\end{bmatrix}.
    \]
    It then follows that $c^Te^{At}b=c_Ve^{A_Vt}b_V$, but note that
    \[
        \Span[b_V,A_Vb_V,\ldots,A_V^{k-1}b_V]=V
    \]
    by construction, therefore $(A_V,b_V)$ is controllable.

    We can now assume that $A$ is stable and $(A,b)$ controllable. Without loss of generality,
    we can assume that $f_c(0)=c^Tb\geqslant 0$, by considering $-c$ instead of $c$ if this is not the case.
    Also recall that $\Reach:=\Reach(A,b,[-1,1])$.
    Therefore, by \Cref{thm:bdry_descrptn},
    $\Reach(A,b,U)$ is an open convex set containing $0$ and its boundary
    is given by $\partial\Reach(A,b,U) = \set*{\beta_v: v \in \RR^n\setminus \{\mathbf{0}\}}$,
    where
    \[
        \beta_v:=\int_0^{\infty}e^{A t}b\sgn(v^Te^{A t}b)\dd t
    \]
    which is a strictly convex set. We start by claiming the following:
    \begin{enumerate}[label=\textbf{(\alph*)}]
        \item\label{claim1} $x\in\overline{\Reach}$ if and only if $x=\int_0^\infty e^{At}bu(t)\dd t$
            for some $u:\RR\to[-1,1]$,
        \item\label{claim2} $\beta_c$ is the unique extermal point in direction $c$:
            $\overline{\Reach}\cap(\beta_c+H_c)=\set{\beta_c}$ where $H_c=\set{x\in\RR^n:c^Tx=0}$
            is the hyperplane of normal $c$,
        \item\label{claim3} $-A^{-1}b\in\overline{\Reach}$,
        \item\label{claim4} $-A^{-1}b=\beta_c$ if and only if $f_c$ has no zero-crossings (nontangential zeros),
        \item\label{claim5} $-A^{-1}b\neq\beta_{-c}$,
        \item\label{claim6} $-A^{-1}b=\beta_c$ if and only if $(-A^{-1}b+H_c)\cap\Reach=\varnothing$ 
    \end{enumerate}
    We now go through those claims one by one: claim~\ref{claim1} is direct consequence of \eqref{eq:sol_lti}
    with the change of variable $\xi=t-s$. Claim~\ref{claim2} is essentially already proven in \cite{HuLQ02}
    but we give the gist of the proof. Pick $x\in\overline{\Reach}$, then by claim~\ref{claim1}
    there exists $u:\RR\to[-1,1]$ such that $x=\int_0^{\infty}e^{At}bu(t)\dd t$. Then check that, since $|u(t)|\leqslant1$,
    \begin{align*}
        c^Tx
            &=\int_0^{\infty}c^Te^{At}bu(t)\dd t
            \leqslant \int_0^{\infty}|c^Te^{At}b|\dd t\\
            &=\int_0^{\infty}c^Te^{At}b\sgn(c^Te^{At}b)\dd t
            =c^T\beta_c.
    \end{align*}
    Therefore, $\beta_c$ is a maximizer in direction $c$. But by strict convexity of $\overline{\Reach}$,
    there can only be a unique one, for otherwise a whole line segment would be in $\partial\Reach$
    which would be a contradiction. To show claim~\ref{claim3}, simply observe that since $e^{At}\to0$
    as $t\to\infty$, we have, by claim~\ref{claim1},
    \[
        -A^{-1}b
            =\int_0^\infty e^{At}b\dd t=\int_0^\infty e^{At}bu(t)\dd t\in\overline{\Reach}
            \qquad\text{where }u(t)=1.
    \]
    Claim~\ref{claim4} follows from the computation above and the remark that
    \begin{align*}
        c^T(\beta_c+A^{-1}b)
            &=\int_0^\infty c^Te^{At}b\big(\sgn(c^Te^{At}b)-1\big)\dd t\\
            &=\int_0^\infty |f_c(t)|-f_c(t)\dd t.
    \end{align*}
    Indeed, there are two cases (since we assumed that $f_c(0)\geqslant0$):
    \begin{itemize}
        \item Either $f_c(t)\geqslant 0$ for all $t\geqslant 0$, then $f_c$ has no zero-crossings and $c^T(\beta_c+A^{-1})b=0$.
            But $-A^{-1}b\in\overline{\Reach}$ and $\beta_c$ is the maximizer in direction $c$,
            so by strict convexity, $\beta_c=-A^{-1}b$.
        \item Either $f_c$ has at least one zero-crossing, then by continuity there exists an open
            interval $(a,b)$ such that $|f_c(t)|-f_c(t)>0$. But since $|f_c(t)|-f_c(t)\geqslant 0$
            for all $t$, it follows that $c^T(\beta_c+A^{-1}b)>0$ hence $\beta_c\neq-A^{-1}b$.
    \end{itemize}
    Claim~\ref{claim5} follows from a similar argument since $\beta_{-c}=-\beta_c$ and
    \begin{align*}
        c^T(\beta_c-A^{-1}b)
            &=\int_0^\infty c^Te^{At}b\big(\sgn(c^Te^{At}b)+1\big)\dd t\\
            &=\int_0^\infty |f_c(t)|+f_c(t)\dd t.
    \end{align*}
    But now assume, for contradiction, that $\beta_{-c}=-A^{-1}b$. Then $c^T(\beta_c-A^{-1}b)=0$
    so $f_c(t)\leqslant 0$ for all $t$, hence $f_c$ has no zero-crossings. But then $-A^{-1}b=\beta_c$
    by claim~\ref{claim4}. This implies that $c^T(\beta_c+A^{-1}b)$ and $f_c(t)\geqslant 0$ by the
    computation of claim~\ref{claim4}. Therefore, $f_c\equiv0$ and $(A,b)$ is not controllable by
    \Cref{lem:control_fc}, a contradiction.

    Finally, for claim~\ref{claim6}, observe that if $-A^{-1}b\in\Reach$ then $-A^{-1}b\neq\beta_c$
    since $\beta_c\in\partial\Reach$ and $\Reach$ is open, so the result is true.
    Otherwise, $-A^{-1}b\in\partial\Reach$ by claim~\ref{claim5}.
    Hence, there are two cases: if $-A^{-1}b=\beta_c$ then the intersection is empty by claim~\ref{claim2}.
    Otherwise, $-A^{-1}b\neq\beta_c$, but $-A^{-1}b\neq\beta_{-c}$ by claim~\ref{claim5} so it must be
    the case that $-A^{-1}b+H_c$ interesects $\Reach$. To show this formally, observe that
    by strict convexity $(\beta_{-c},\beta_c)\subseteq\Reach$. Also, since $\beta_c$
    is the maximizer in direction $c$, $-A^{-1}b\neq\beta_c$ implies that $c^Tx<c^T\beta_c$ where $x=-A^{-1}b$.
    A similar reasoning for $\beta_{-c}$ shows that $c^T\beta_{-c}<c^Tx<c^T\beta_c$. Define
    $y=\beta_{-c}+\alpha(\beta_c-\beta_{-c})$ where $\alpha=\tfrac{c^Tx-c^T\beta_{-c}}{c^T\beta_c-c^T\beta_{-c}}$.
    Then $\alpha\in(0,1)$ by the previous inequality, hence $y\in(\beta_{-c},\beta_c)\subseteq\Reach$.
    But by construction $c^Ty=c^T\beta_{-c}+\alpha c^T(\beta_c-\beta_{-c})=c^Tx$, therefore $y\in x+H_c$.

    We can now describe an algorithm that solves this instance of Continuous Nontangential Skolem problem
    by reducing to the LTI Set Reachability problem. Consider the LTI with matrix $A$, control matrix $b$,
    input set $[-1,1]$ and target set $Y=\set{x\in\RR^n:c^Tx=-c^TA^{-1}b}$. Note that $Y$ is a hyperplane
    with algebraic coefficients, and check that $Y=-A^{-1}b+H_c$. Let $\Reach$ be the reachable set of
    this LTI, then it follows by claims~\ref{claim3}~and~\ref{claim6}
    above that $Y\cap\Reach=\varnothing$ if and only if $-A^{-1}b=\beta_c$ if and only if
    $f_c$ has no zero-crossings. Hence, $Y$ is reachable if and only if the Continuous Nontangential Skolem
    problem instance has a zero-crossing.

    The set $Y$ above is convex but not compact, but since we have chosen $A$ to be stable, the reachable
    set is bounded by \Cref{prop:null_controllable_region_desc} and it is easy to compute a
    bound $M$ such that $\Reach\subseteq[-M,M]^n$. Then we can define $\hat{Y}=Y\cap[-M,M]^n$
    which is now compact convex, and clearly $\Reach\cap Y=\varnothing$ if and only if
    $\Reach\cap\hat{Y}=\varnothing$.
\end{proof}

\begin{acks}
    We thank James Worrell for useful discussions on the paper.
\end{acks}

\bibliographystyle{ACM-Reference-Format}
\bibliography{biblio}

\ifthenelse{\boolean{fullversion}}{
    \appendix
    \section{Additional preliminaries}\label{apx:prelim}

\subsection{Jordan decomposition and matrix exponential}

Given a square matrix $A$ of order $n$ with rational entries,
one can find matrices $P$ and $\Lambda$ (possibly with complex algebraic entries)
such that $A = P \Lambda P^{-1}$.
Here $\Lambda$ is a block diagonal matrix $\diag(J_1,J_2,\dots,J_m)$ where
the $J_i$ are matrices of a special form (given below) known as the Jordan blocks.
A particular application of Jordan decomposition is to compute the exponential of a matrix.
From the above definition, it is clear that if $A = P \Lambda P^{-1}$, then $e^{At} = P e^{\Lambda t} P^{-1}$.
If $\Lambda = \diag(J_1,J_2,\ldots,J_m)$, then it is not hard to see that $e^{\Lambda t} = \diag(e^{J_1 t},e^{J_2 t},\ldots,e^{J_m t})$.
A closed form expression for a Jordan block $e^{J_i t}$ is given by
\[
    e^{J_i t} = e^{\lambda_i t}\begin{bmatrix}
    1 & t & \frac{t^2}{2} & \cdots & \frac{t^{k-1}}{(k-1)!} \\
    0 & 1 & t & \cdots & \frac{t^{k-2}}{(k-2)!} \\
    \vdots & \vdots & \ddots & \ddots & \vdots \\
    0 & 0 & \cdots & 1 & t \\
    0 & 0 & \cdots & 0 & 1
    \end{bmatrix},
    J_i = \begin{bmatrix}
        \lambda_i & 1 & 0 & \cdots & 0 \\
        0 & \lambda_i & 1 & \cdots & 0 \\
        \vdots & \vdots & \ddots & \ddots & \vdots \\
        0 & 0 & \cdots & \lambda_i & 1 \\
        0 & 0 & \cdots & 0 & \lambda_i
        \end{bmatrix}
\]
where $\lambda_i$ is an eigenvalues of $A$ and $k$ is the size of the Jordan block. A consequence
of this normal form is the \emph{real Jordan normal form}: if $A$ is real then its Jordan form
can be nonreal. However, one can allow more general blocks to recover a real representation: a real
Jordan block is either a complex Jordan block with a real $\lambda_i$, or a block matrix of the form
\[
    J_i'=\begin{bmatrix}C_i&I_2&0&\cdots&0\\0&C_i&I_2&\cdots 0\\\vdots&\vdots&\ddots&\ddots&\vdots\\
            0&0&\cdots&C_i&I_2\\0&0&\cdots&0&C_i\end{bmatrix}
    \qquad\text{where}\qquad
    C_i=\begin{bmatrix}a_i&-b_i\\b_i&a_i\end{bmatrix}
\]
where $\lambda_i=a_i+ib_i$. In this case, one can ensure that the transformation matrix $P$ is also real.
This form is particularly useful for simple blocks since the exponential of $C_i$ is a scaling-and-rotate
matrix:
\[
    e^{C_it}=e^{a_it}\begin{bmatrix}\cos(b_it)&\sin(b_it)\\-\sin(b_it)&\cos(b_it)\end{bmatrix}.
\]

\section{Proof of Lemma~\ref{prop:reduce_controllable_minkowski}}\label{apx:reduce_controllable_minkowski}

Following the observations of \Cref{sec:control_theory} and in particular \eqref{eq:decompose_C_hypercube},
we have that $\Reach_\tau(A,B,[-1,1]^m)=\sum_{i=1}^m\Reach_\tau(A,b_i,[-1,1])$. We now focus on the case
where $B=b$ is a column vector. If $(A,b)$ is controllable then there is nothing to do. Otherwise
let $V=\Span(b,Ab,\ldots,A^{n-1}b)$, and assume that $k:=\dim V<n$.
Then $b\in V$ and $AV\subseteq V$, so by a change of basis $P$ sending $V$ to $\RR^k$ we have
\[
    P^{-1}AP=\begin{bmatrix}A_V&0\\0&*\end{bmatrix},\qquad
    P^{-1}b=\begin{bmatrix}b_V\\0\end{bmatrix}.
\]
Then $\Span(b_V,A_Vb_V,\ldots,A_V^{k-1}b_V)=\RR^k$ by construction, therefore $(A_V,b_V)$ is controllable.
Furthermore, it is clear that the spectrum of $A_V$ is included in that of $A$.
On the other hand, for any input $u$,
\begin{align*}
    \int_0^\tau e^{At}bu(t)\dd t
        &=P\int_0^\tau \begin{bmatrix}e^{A_Vt}&*\\0&*\end{bmatrix}\begin{bmatrix}b_V\\0\end{bmatrix}u(t)\dd t\\
        &=P\begin{bmatrix}
            \int_0^\tau e^{A_Vt}b_Vu(t)\dd t\\0
        \end{bmatrix}.
\end{align*}
It follows that
\[
    \Reach_\tau(A,b,[-1,1])=PJ_k\Reach(A_V,b_V,[-1,1])
    \quad\text{where}\quad
    J_k=\begin{bmatrix}I_k\\0\end{bmatrix}\in\RR^{n\times k}.
\]
Going back to the general case, we now have that
\[
    \Reach_\tau(A,B,[-1,1]^m)=\sum_{i=1}^mP_i\Reach_\tau(C_i,b_i,[-1,1]).
\]
Assume there are formulas $\Phi_1,\ldots,\Phi_m$ in some theory $\RRtheory$
to expressible membership in $\Reach_\tau(C_i,b_i,[-1,1])$. Then we can express membership of some $y\in\RR^n$
in $\Reach_\tau(A,B,[-1,1]^m)$ by the formula
\[
    \Phi(y):=\exists z_1,\ldots,\exists z_k.\,\Phi_1(z_1)\wedge\cdots\wedge\Phi_m(z_m)\wedge y=P_1z_1+\cdots+P_mz_m.
\]
Clearly if the theory $\RRtheory$ contains $\RRtarski$ then $\Phi$ is in $\RRtheory$. If instead of a
formula $\Phi_i$ for $\Reach_\tau(C_i,b_i,[-1,1])$, we have a formula for its boundary $\partial\Reach_\tau(C_i,b_i,[-1,1])$,
then we note that $\Reach_\tau(C_i,b_i,[-1,1])$ is open convex by \Cref{prop:null_controllable_region_desc}
when $\tau=\infty$ and convex closed when $\tau<\infty$,
hence we can write a formula for $\Reach(C_i,b_i,[-1,1])$ from $\Phi_i$, as shown below.

Let $C\subseteq\RR^n$ be an open (resp. closed) bounded convex set, if we have a formula $\Phi$ to express
membership in $\partial C$ in some theory $\RRtheory$ that subsumes $\RRtarski$, then we can write
a formula in $\RRtheory$ to express membership in $C$.
Indeed, by Krein–Milman theorem, the closure $\overline{C}$ of $C$ is the convex hull of its extreme points,
but the extreme points of $\overline{C}$ are on the boundary $\partial C$. Hence $\overline{C}$
is the convex hull of $\partial C$. It follows by Carathéodory's theorem that any point in $\overline{C}$
is the convex combination of at most $n+1$ in $\partial C$. Hence we can write a formula $\psi$ to
decide membership in $\overline{C}$. If $C$ is closed then $C=\overline{C}$ so we are done.
If $C$ is open, we know that $C=\overline{C}\setminus\partial C$,
hence we can write a formula for $C$.

\section{Proof of Proposition~\ref{prop:complement_border_reach_lti_approx}}\label{apx:complement_border_reach_lti_approx}

We first observe that we can reduce to the case where $B=b$ is a column vector and $(A,b)$
is controllable. Indeed, apply \Cref{prop:reduce_controllable_minkowski} to get computable $C_1,\ldots,C_k$
and $P_1,\ldots,P_k$ such that
\[
    \Reach(A,B,[-1,1]^m)=\sum_{i=1}^mP_i\Reach(C_i,b_i,[-1,1]).
\]
Now assume that we have some convex under/over-approximation $Q_i^-,Q_i^+$ of
$\Reach(C_i,b_i,[-1,1])$ for all $i$. Then $\sum_{i=1}^mP_iQ_i^\pm$ is a convex under/over-approximation
by the property of the Minkowski sum of convex sets, furthermore this sum can be computed effectively.
Note that in this reduction, the matrix $A$ has not changed, hence it is still stable.

We now focus on the case where $B=b$ is a column vector such that $(A,b)$ is controllable.
We can apply \Cref{thm:bdry_descrptn} to get that $\Control(A,b,U)$ is an open convex set
containing $\boldsymbol{0}$ and
\[
    \partial\Reach(A,b,U) = \set*{\int_0^{\infty} e^{A t}b\sgn(c^Te^{A t}b)\dd t: c \in \RR^n\setminus \{\mathbf{0}\}}.
\]
Observe that only the direction of $c$ matters so we can restrict the set to the \emph{compact}
subset of $c$ such that $\twonorm{c}=1$.
Let $f_c(t)=e^{A t}b\sgn(c^Te^{A t}b)$ and observe that since $A$ is stable,
$f_c(t)\to0$ as $t\to\infty$. In fact, one can compute constants $D$ and $\alpha<0$ such that
$\|f_c(t)\|\leqslant De^{\alpha t}$ for all $t\geqslant0$.
Let $T$ to be fixed later,
then
\[
    \left\|\int_{0}^\infty f_c(t)\dd t-\int_0^Tf_c(t)\dd t\right\|\leqslant
        \int_T^\infty De^{\alpha t}\dd t=D\alpha^{-1}e^{\alpha T}.
\]
Furthermore, one can approximate $\int_0^Tf_c(t)\dd t$ with arbitrary precision given $T$ and $c$,
by using the fact that $c^Te^{At}b$ is analytical and computable. Furthermore, $c\mapsto \int_0^Tf_c(t)\dd t$
is continuous since the zero-crossings of $c^Te^{At}B$ move continuously with $c$ and the
(discontinous) tangential zeros that can appear do not change the integral. It follows that
on the compact set $\set{c:\twonorm{c}=1}$, it has bounded variations, with a computable bound.
Putting everything together, this allows us to sample the border with sufficiently many points
as to obtain an underapproximation and overapproximation of the border, in the form of a convex set.

\section{Proof of Proposition~\ref{prop:decide_real_eigen_rexp}}\label{apx:decide_real_eigen_rexp}

By putting $A$ in Jordan Normal Form, write $A=Q^{-1}MQ$ where $Q$ is invertible
and $M$ is made of Jordan blocks. Since $A$ has real spectrum, $Q$ and $M$ are real matrices
and $\Reach(A,B,U)=Q^{-1}\Reach(M,QB,U)$, we can now assume that $A$ only consists of Jordan blocks
since $y\in\Reach(A,B,U)$ if and only if $Qy\in\Reach(M,QB,U)$.

Assume that $A=\diag(J_1,\ldots,J_k)$ consist of Jordan blocks. Without loss of generality, we can assume that
the blocks are ordered by increasing eigenvalue. Hence we can write $A=\diag(A_1,A_2)$ where $A_1$
contains the nonnegative $\lambda_i$ and $A_2$ the negative ones. Then $A_1$ is weakly-antistable and
$A_2$ is stable. Decompose $B$ into $B_1$ and $B_2$ accordingly. Then by \Cref{prop:null_controllable_region_desc},
$\Reach(A,B,U)=\RR^{n_1}\times \Reach(A_2,B_2,U)$. We can then apply \Cref{prop:reduce_controllable_minkowski}
to decompose $\Reach(A_2,B_2,U)$ into smaller controllable problems $(C_i,b_i)$, where each $C_i$
also has a real spectrum. It then suffices to show that membership in $\partial\Reach(C_i,b_i,[-1,1])$
is expressible in $\RRexp$ for each subproblem to conclude by \Cref{prop:reduce_controllable_minkowski}.

Assume that $A$ only has negative eigenvalues, $B=b$ is a column vector, $(A,b)$ is controllable and $U=[-1,1]$.
Then by \Cref{thm:bdry_descrptn} we have that
$\partial\Reach(A,b,[-1,1]) = \set*{\beta_c: c \in \RR^n\setminus \{\mathbf{0}\}}$, where
\[
    \beta_c:=\int_0^{\infty}e^{At}b\sgn(c^Te^{A t}b)\dd t
\]
But observe that $f_c(t):=c^Te^{A t}b$ is an exponential polynomial in $t$. Furthermore,
it has at most $n-1$ zeros by \Cref{lem:control_fc} since $A$ has real eigenvalues and $b,c$
are nonzero. Let $0\leqslant k\leqslant n-1$ and assume that $f_c$ has $k$ nontangential zeros
(the tangential zeros will play no role) $0<t_1<\cdots<t_k$. Then, depending
on the sign of $f_c(0)$, and noting that the sign of $f_c$ changes at each $t_i$, we have
that
\begin{align*}
    \beta_c
        &=\pm\left(\int_0^{t_1}e^{-At}b\dd t+\sum_{i=1}^{k-1}\int_{t_i}^{t_{i+1}}(-1)^i\dd t
            +(-1)^k\int_{t_k}^{\infty}e^{-At}b\dd t\right)\\
        &=\pm A^{-1}\left(I_n+2\sum_{i=1}^k(-1)^ie^{At_i}\right)b\\
        &=\pm R_k(t_1,\ldots,t_k)
\end{align*}
which is also a (vector of) exponential polynomials. We can now
write a formula in $\RRexp$ to express that a target $y$ is on the border:
\begin{align*}
    \Phi(y)
        &:=\exists c.c\neq0\bigwedge\bigvee_{k=0}^d\Phi_k(y,c)
    \intertext{to check for a point on a border in direction $c$,}
    \Phi_k(y,c)
        &:=\exists t_1,\ldots,t_k.\Psi_k(c,t)\bigwedge\Psi_k'(c,t)\bigwedge
            \bigwedge_{j=1}^n\left(y_j=R_{k,j}(t)\right)
    \intertext{to match the target with some parameters $z_1,\ldots,z_k$,}
    \Psi_k(c,t)
        &:=\bigwedge_{i=1}^k \left(f_c(t_i)=0\wedge f_c'(t_i)\neq 0\right)
            \bigwedge 0<t_1<\cdots<t_k
    \intertext{to check that the $t_i$ are zeros of $Q(c,\cdot)$,}
    \Psi_k'(c,t)
        &:=\forall u.(u>0\wedge f_c(u)=0\wedge f_c'(u)\neq 0)\Rightarrow\bigvee_{i=1}^k u=z_i
    \intertext{to check the $t_i$ are the only zeros of $Q(c,\cdot)$.}
\end{align*}
We note that those are indeed formulas in $\RRexp$ because $f_c,f_c'$ and $e^{At}$ are exponential
polynomials in $t$ and $c$ and they can be computed by putting $A$ in Jordan normal form, and all
those exponential polynomials have algebraic coefficients since $b$ and $A$ have algebraic coefficients.
This shows that deciding the border reduces to deciding a formula in $\RRexp$.

\section{Proof of Proposition~\ref{prop:decide_complex_eigen_2d_rexpsin}}\label{apx:decide_complex_eigen_2d_rexpsin}

If $A$ has real eigenvalues, then the decidability reduces to $\RRexp$ by \Cref{prop:decide_real_eigen_rexp},
which is decidable in $\RRexpsin$. Otherwise, since $A$ is real, it must have two conjugate complex
but nonreal eigenvalues. Hence, we can put $A$ in \emph{real} Jordan Form: $A=P^{-1}JP$ where $P$
is real and
\[
    J=\begin{bmatrix}\lambda&\theta\\-\theta&\lambda\end{bmatrix},\qquad
\]
where $\lambda,\theta\in\RR$. It follows that $\Reach(A,B,U)=P^{-1}\Reach(J,PB,U)$
so we now focus on this particular case.
Assume that $A$ is a real Jordan block of the form above. Then
\[
    e^{Jt}=e^{\lambda t}\begin{bmatrix}\cos(\theta t)&\sin(\theta t)\\-\sin(\theta t)&\cos(\theta t)\end{bmatrix}.
\]
Furthermore, we can decompose $B$ into columns vectors $b_1,\ldots,b_m$ so that
\[
    \Reach(A,B,[-1,1]^m)=\Reach(A,b_1,[-1,1])+\cdots+\Reach(A,b_m,[-1,1])
\]
so we now assume that $B=b$ is a column vector. We further reduce to the case where $b$ has
norm $1$ by noticing that $\Reach(A,b_1,[-1,1])=\twonorm{b}\Reach(A,b/\twonorm{b},[-1,1])$.
Now $(A,b)$ is controllable (unless $b=0$ which is trivial)
since $A$ rotates (and rescale) $b$ by an angle $\theta$ and $\theta\neq 0\pmod \pi$ (indeed,
$\theta$ is nonzero and algebraic). Hence we can apply \Cref{prop:null_controllable_region_desc}
and \Cref{thm:bdry_descrptn} to get that either $\Reach(A,b,[-1,1])=\RR^2$ if $\lambda>0$, or
$\partial\Reach(A,B,[-1,1]^m)=\set{\beta_c:c\in\RR^2\setminus\set{0}}$ where
\[
    \beta_c:=\int_0^{\infty}e^{At}b\sgn(c^Te^{A t}b)\dd t.
\]
We now focus on this case since the other one is trivial. Since the dimension is two, we can write
$c_\phi:=\begin{bmatrix}\cos(\phi)&\sin(\phi)\end{bmatrix}^T$ for $\phi\in[0,2\pi)$ and
$b=\begin{bmatrix}\cos\beta&\sin\beta\end{bmatrix}^T$.
Then,
\[
    \begin{bmatrix}\cos(\theta t)&\sin(\theta t)\\-\sin(\theta t)&\cos(\theta t)\end{bmatrix}
        \begin{bmatrix}a\\b\end{bmatrix}
    =\begin{bmatrix}\cos(\theta t+\beta)\\\sin(\theta t+\beta\end{bmatrix}.
\]
and
\begin{align*}
    \sgn(c^Te^{A t}b)
        &=\sgn\left(\begin{bmatrix}\cos(\phi)&\sin(\phi)\end{bmatrix}
        \begin{bmatrix}\cos(\theta t+\beta)\\\sin(\theta t+\beta\end{bmatrix}\right)\\
        &=\sgn(\cos(\theta t+\beta+\phi)).
\end{align*}
Hence,
\begin{align*}
    \beta_\phi
        &:=\beta_{c_\phi}=\int_0^\infty e^{\lambda t}\begin{bmatrix}\cos(\theta t+\beta)\\\sin(\theta t+\beta\end{bmatrix}
        \sgn(\cos(\theta t+\beta+\phi))\dd t\\
        &=e^{i\beta}\int_0^te^{(\lambda+i\theta)t}\sgn(\cos((\theta t+\beta+\phi))\dd t
\end{align*}
when viewed as a complex number (to simplify computations). Let $t_{\phi}$ denote the smallest
$t\geqslant 0$ such that $\cos(\theta t+\beta+\phi)=0$ and $\varepsilon_\phi=\sgn(\cos(\beta+\phi))$,
except when $t_\phi=0$ in which case we let $\varepsilon_\phi=-1$.
Then
\begin{align*}
    \beta_\phi
        &=e^{i\beta}\varepsilon_\phi\left(\int_0^{t_\phi}e^{(\lambda+i\theta )t}\dd t
        -\sum_{k=0}^\infty\int_{t_\phi+k\tfrac{\pi}{\theta}}^{t_\phi+(k+1)\tfrac{\pi}{\theta}}
        (-1)^ke^{(\lambda+i\theta )t}
        \dd t\right)\\
        &=\frac{e^{i\beta}\varepsilon_\phi}{\lambda+i\theta}\Big(e^{(\lambda+i\theta )t_\phi}-1\\
        &\hspace{1.5cm}-\sum_{k=0}^\infty(-1)^k\left(e^{(\lambda+i\theta )(t_\phi+(k+1)\tfrac{\pi}{\theta})}
        -e^{(\lambda+i\theta )(t_\phi+(k+1)\tfrac{\pi}{\theta})}\Big)
        \right)\\
        &=\frac{e^{i\beta}\varepsilon_\phi}{\lambda+i\theta}\left(-1+
        2\sum_{k=0}^\infty(-1)^ke^{(\lambda+i\theta )(t_\phi+k\tfrac{\pi}{\theta})}
        \right)\\
        &=\frac{e^{i\beta}\varepsilon_\phi}{\lambda+i\theta}\left(-1+
        2e^{(\lambda+i\theta)t_\phi}\sum_{k=0}^\infty(-1)^ke^{k(\lambda+i\theta )\tfrac{\pi}{\theta}}
        \right)\\
        &=\frac{e^{i\beta}\varepsilon_\phi}{\lambda+i\theta}\left(-1+
        2e^{(\lambda+i\theta)t_\phi}\frac{1}{1+e^{(\lambda+i\theta)\tfrac{\pi}{\theta}}}\right)\\
        &=\frac{e^{i\beta}\varepsilon_\phi}{\lambda+i\theta}\left(-1+
        2e^{(\lambda+i\theta)t_\phi}\frac{1}{1-e^{\tfrac{\lambda\pi}{\theta}}}\right).
\end{align*}
One can then obtain an expression for each coordinate of $\beta_\phi$ viewed as a 2D vector.
In particular, this expression is expressible in $\RRexpsin$ where $\sin$ is taken over some
bounded interval. Indeed, $\lambda,\beta$ and $\theta$ are algebraic, $t_\phi$ can be express in with an equation
involving a bounded $\sin$ since we have the trivial bound $t_\phi\leqslant \tfrac{2\pi}{\theta}$.
We can then express $e^{(\lambda+i\theta)t_\phi}$ using a combination of exponential and bounded $\sin$,
again noting that $\theta t_\phi\leqslant 2\pi$. We can also express $\pi$ using an equation on
bounded $\sin$ ($\sin{\pi}=0 \wedge (\forall y.\,0<y<\pi \implies \sin{y}\neq 0)$) hence we can define
$e^{\tfrac{\lambda\pi}{\theta}}$. It follows that we can express $\beta_\phi$ and hence the boundary
in $\RRexpsin$. We then conclude using \Cref{prop:reduce_controllable_minkowski}.

\section{Proof of Proposition~\ref{prop:decide_2d_real_eigen}}\label{apx:decide_2d_real_eigen}

If $A$ is diagonalizable, and since it has real eigenvalues, we can write $A=P^{-1}DP$ where $P$ is real and $D$ is real diagonal.
Then $\Reach(A,B,U)=P^{-1}\Reach(D,PB,U)$ so we only need to decide if $Py\in\Reach(D,PB,U)$.
But in dimension $2$ all columns of $PB$ necessarily have at most two nonzero entries
so we can conclude using \Cref{prop:decide_real_mat_uncond}.
Otherwise, $A$ only has one eigenvalue, which is real, so we conclude with \Cref{prop:decide_real_mat_uncond}.

\section{Proof of Proposition~\ref{prop:decide_bounded_time_rexpsin}}\label{apx:decide_bounded_time_rexpsin}

Recall that we are concerned with a time-bounded problem, for some time bound $\tau<\infty$
that is algebraic.
Since $U=[-1,1]^m$, we can apply \Cref{prop:reduce_controllable_minkowski} to 
to decompose $\Reach(A,B,U)$ into smaller controllable problems $(C_i,b_i)$, where each $C_i$
is still a real matrix. It then suffices to show that membership in $\partial\Reach(C_i,b_i,[-1,1])$
is expressible in $\RRexpsin$ for each subproblem to conclude by \Cref{prop:reduce_controllable_minkowski}.

Now consider the case where $B=b$ is a column vector and $(A,b)$ is controllable and $U=[-1,1]$.
Since $U$ is convex closed and $\tau$ is finite, it is clear that $\Reach$ is convex closed.
A simple proof similar to that of \Cref{thm:bdry_descrptn} shows that
$\partial\Reach_\tau(A,B,U)=\set*{\beta_c: c \in \RR^n\setminus \{\mathbf{0}\}}$, where
\[
    \beta_c:=\int_0^{\tau}e^{At}b\sgn(c^Te^{A t}b)\dd t
\]
Indeed, observe that for any $c\in\RR^n\setminus \{\mathbf{0}\}$,
\begin{align*}
    c^T\int_0^{\tau} e^{A t}b\sgn(c^Te^{A t}b)\dd t
        &=\int_0^{\tau} c^Te^{A t}b\sgn(c^Te^{A t}b)\dd t\\
        &=\int_0^{\tau} |c^Te^{A t}b|\dd t
\end{align*}
and for any $u:[0,\tau]\to[-1,1]$ measurable,
\begin{align*}
    c^T\int_0^\tau e^{A t}bu(t)\dd t
        &=\int_0^\tau c^Te^{A t}bu(t)\dd t\\
        &\leqslant\int_0^\tau |c^Te^{A t}bu(t)|\dd t
        \leqslant\int_0^\tau |c^Te^{A t}b|\dd t.
\end{align*}
Hence, $\int_0^{\tau} e^{A t}b\sgn(c^Te^{A t}b)\dd t$ is a maximizer in direction $c$ and we conclude
by standard results about convex sets. But observe that $f_c(t):=c^Te^{-A t}b$ is an exponential polynomial in $t$.
Furthermore, it has a bounded number of zeroes by \Cref{lem:control_fc} since $b,c$ are nonzero.
We can even obtain an explicit formula for this bound by induction on the number of terms, or invoke
a more general result on zeroes in a complex ball by Tijdeman \cite{TIJDEMAN19711}:

\begin{lemma}
    If $f(z) = \sum_{k=1}^l P_k(z)e^{\lambda_k z}$ where $P_k$ is a polynomial of degree $\rho_k-1$,
    the number of zeros in the complex plane in any ball of radius $R$ is bounded by $3(n_0-1)+4R\Delta$
    where $n_0 = \sum_{k=1}^l \rho_k$ and $\Delta = max_{k} |\lambda_k|$. \\
\end{lemma}

In our case, the bound $N$ on the number of zeroes is clearly computable from $A$. Let $c\neq\mathbf{0}$,
$0\leqslant k\leqslant N$ and assume that $f_c$ has $k$ nontangeantial zeros
(the tangeantial zeros will play no role) $0<t_1<\cdots<t_k$ in $[0,\tau]$. Then, depending
on the sign of $f_c(0)$, and noting that the sign of $f_c$ changes at each $t_i$, we have that
\begin{align*}
    \beta_c
        &=\pm\left(\int_0^{t_1}e^{-At}b\dd t+\sum_{i=1}^{k-1}\int_{t_i}^{t_{i+1}}(-1)^i\dd t
            +(-1)^k\int_{t_k}^{\tau}e^{-At}b\dd t\right)\\
        &=\pm A^{-1}\left(I_n+2\sum_{i=1}^k(-1)^ie^{At_i}-(-1)^ke^{-A\tau}\right)b\\
        &=\pm R_k(t_1,\ldots,t_k)
\end{align*}
which is also a (vector of) exponential polynomials. We can now proceed as in the proof of \Cref{prop:decide_real_eigen_rexp}
and write a formula in $\RRexpsin$ to express that a target $y$ is on the border. The crucial
point is that the exponentials in $e^{At_i}$ will not only involve the real exponential but also
some sine and cosine. This is where the boundedness is crucial: since $t_i\leqslant\tau$, we only
need bounded sine and cosine in our formulas. Also note that since $\tau$ is algebraic, we can
write $\tau$ in the formulas (this is necessary to express $e^{-A\tau}$); we note that since we are
working in $\RRexpsin$, we could in fact allow more general $\tau$ than just algebraic numbers.
This shows that deciding the border reduces to deciding a formula in $\RRexpsin$.

\section{Proof of Theorem~\ref{th:lti_hard_skolem}}\label{apx:lti_hard_skolem}

Let $c,A,b$ be an instance of the Continuous Skolen problem. Let $f(t)=c^Te^{At}b$,
the problem asks whether $f$ has any zero at $t\geqslant 0$. Without loss of generality,
we can assume that $c^Tb\geqslant0$ by changing $b$ into $-b$.
Now let $u=(1,\ldots,1)$, $B\in\QQ^{n\times n}$ be such that $Bu=b$ and $U=\set{u}$, and observe that by definition,
\[
    \Reach(A,AB,U)
        =\set*{\int_0^Te^{At}ABu\dd t:T\geqslant 0}
        =\set*{(e^{At}-I_n)b:t\geqslant 0}.
\]
But notice that for any $t\geqslant 0$,
\[
    c^T(e^{At}-I_n)b
        =c^Te^{At}b-c^Tb
        =f(t)-c^Tb.
\]
Hence if we define the set $Y=\set{x\in\RR^n:c^Tx\leqslant -c^Tb}$, which is a hyperplane, then
$Y$ is reachable if and only if there exists $t\in\RR^n$ such that $f(t)\leqslant 0$.
But since $f(0)=c^Tb\geqslant 0$ by assumption, this last condition is equivalent to the existence
of a zero by continuity. This shows that Skolem instance $(c,A,b)$ is positive if and only the
Set Reachability instance $(A,AB,U,Y)$ is positive.

Note that we can easily modify the instance to further strengthen the result like in the proof
of \Cref{th:lti_nontangen_hard_skolem}. More precisely, we can ensure that the LTI instance is
stable and that the set $Y$ is compact convex.

\section{Details on Figure~\ref{fig:ex_reach2d}}\label{sec:details_fig_reach}

Since all eigenvalues of $A$ are negative, it is stable. Furthermore, one checks that $b_1$ and $Ab_1$
are linearly independent, hence $(A,b_1)$ is controllable. We can apply \Cref{thm:bdry_descrptn}
to get the description of the boundary:
\begin{align*}
    \partial\Reach(A,b_1)
        &=\set*{\int_0^{\infty} e^{A t}b\sgn(c^Te^{A t}b)\dd t: c \in \RR^2\setminus \{\mathbf{0}\}}\\
        &=\Bigg\{\int_0^{\infty} \begin{bmatrix}e^{-\frac{t}{2}}&0\\0&e^{-\tfrac{t}{3}}\end{bmatrix}
            \begin{bmatrix}1\\1\end{bmatrix}\times\\
        &\hspace{0.5cm}\sgn\left(\begin{bmatrix}c_1&c_2\end{bmatrix}\begin{bmatrix}e^{-\frac{t}{2}}&0\\0&e^{-\tfrac{t}{3}}\end{bmatrix}
            \begin{bmatrix}1\\1\end{bmatrix}\right)\dd t: c_1,c_2 \in \RR\setminus \{0\}\Bigg\}\\
        &=\set*{\int_0^{\infty}\begin{bmatrix}e^{-\tfrac{t}{2}}\\e^{-\tfrac{t}{3}}\end{bmatrix}
            \sgn\left(c_1e^{-\tfrac{t}{2}}+c_2e^{-\tfrac{t}{3}}\right)\dd t: c_1,c_2 \in \RR\setminus \{0\}}.
\end{align*}
Given $c_1,c_2$ nonzero, observe that
\[
    c_1e^{-\tfrac{t}{2}}+c_2e^{-\tfrac{t}{3}}=0
    \;\Leftrightarrow\;
    1+\tfrac{c_2}{c_1}e^{\tfrac{t}{6}}=0
    \;\Leftrightarrow\;
    t=6\ln\tfrac{-c_1}{c_2}.
\]
Since only the ratio $c_1/c_2$ is important and $c_2$ must be nonzero, we can parametrize it as
$c_2=1$ and $c_1=-\alpha$, where $\alpha\in(1,+\infty)$, so that the only possible solution becomes
$t=t_1:=6\ln \alpha\geqslant0$.
Now there are two cases:
\begin{itemize}
    \item if $c_1/c_2\geqslant -1$, then there is no solution and the integral becomes
        \[
            \int_0^{\infty}\begin{bmatrix}e^{-\tfrac{t}{2}}\\e^{-\tfrac{t}{3}}\end{bmatrix}\sgn(c_1+c_2)\dd t
                =\begin{bmatrix}2\\3\end{bmatrix}\sgn(c_1+c_2);
        \]
    \item  if $c_1/c_2\leqslant -1$, then we can split the integral into two parts (the sign must change):
        \begin{align*}
            \int_0^{t_1}&\begin{bmatrix}e^{-\tfrac{t}{2}}\\e^{-\tfrac{t}{3}}\end{bmatrix}\sgn(c_1+c_2)\dd t
                -\int_{t_1}^{\infty}\begin{bmatrix}e^{-\tfrac{t}{2}}\\e^{-\tfrac{t}{3}}\end{bmatrix}\sgn(c_1+c_2)\dd t\\
                &=\begin{bmatrix}2-4e^{-t_1/2}\\3-6e^{-t_1/3}\end{bmatrix}\sgn(c_1+c_2)\\
                &=\begin{bmatrix}2-4e^{-3\ln\alpha}\\3-6e^{-2\ln\alpha}\end{bmatrix}\sgn(c_1+c_2)\\
                &=\begin{bmatrix}2-4\alpha^{-3}\\3-6\alpha^{-2}\end{bmatrix}\sgn(c_1+c_2).
        \end{align*}
\end{itemize}
Hence the boundary of the reachable set is
\[
    \partial\Reach(A,b_1)=\set*{\pm\begin{bmatrix}2-4\alpha^{-3}\\3-6\alpha^{-2}\end{bmatrix}:\alpha\in[1,\infty)}
        \cup\set*{\pm\begin{bmatrix}2\\3\end{bmatrix}}.
\]
The second reachable set is symmetric with respect to the vertical axis, hence we obtain
\[
    \partial\Reach(A,b_2)=\set*{\pm\begin{bmatrix}2-4\alpha^{-3}\\6\alpha^{-2}-3\end{bmatrix}:\alpha\in[1,\infty)}
        \cup\set*{\pm\begin{bmatrix}2\\-3\end{bmatrix}}.
\]
In order to describe the border of $\partial\Reach(A,B)$, we consider the following
question: given a nonzero vector $\tau$, find the (unique by strict convexity) maximizer
in $\partial\Reach(A,b_1)$ in direction $\tau$:
\[
    x_\tau:=\argmax_{x\in \partial\Reach(A,b_1)}\Angle{x,\tau}.
\]
Observe that we have in fact already computed this point because the description of the border by
\Cref{thm:bdry_descrptn} is in the form of maximizers. If we write $\tau=(c_1,c_2)$ 
as above, then
\[
    x_\tau=\begin{cases}
        \begin{bmatrix}2+4(c_1/c_2)^{-3}\\3-6(c_1/c_2)^{-2}\end{bmatrix}\sgn(c_1+c_2)
            &\text{ if }c_1/c_2\leqslant -1\\
        \begin{bmatrix}2\\3\end{bmatrix}\sgn(c_1+c_2)&\text{ otherwise}.
        \end{cases}.
\]
A similar computation for $\partial\Reach(A,b_2)$ gives
\[
    x_\tau'=\begin{cases}
        \begin{bmatrix}2+4(c_1/c_2)^{-3}\\6(c_1/c_2)^{-2}-3\end{bmatrix}\sgn(c_1-c_2)
            &\text{ if }c_1/c_2\geqslant 1\\
        \begin{bmatrix}2\\-3\end{bmatrix}\sgn(c_1-c_2)
            &\text{ otherwise}
        \end{cases}.
\]
Noting that $\sgn(c_1+c_2)=\sgn(c_1-c_2)$ whenever $|c_1/c_2|\geqslant1$, 
the maximizer $y_\tau$ for $\partial\Reach(A,B)$ is
\begin{align*}
    y_\tau=x_\tau+x_\tau'
        &=\begin{cases}
            \begin{bmatrix}4+4(c_1/c_2)^{-3}\\-6(c_1/c_2)^{-2}\end{bmatrix}\sgn(c_1+c_2)
            &\text{if }c_1/c_2\leqslant -1\\
            \begin{bmatrix}4-4(c_1/c_2)^{-3}\\6(c_1/c_2)^{-2}\end{bmatrix}\sgn(c_1+c_2)
            &\text{otherwise}
        \end{cases}\\
        &=\begin{bmatrix}4-4(c_1/c_2)^{-3}\sgn(c_1/c_2)\\6(c_1/c_2)^{-2}\sgn(c_1/c_2)\end{bmatrix}\sgn(c_1+c_2)
\end{align*}
for $|c_1/c_2|\geqslant1$.

\section{Details on the second example}\label{sec:details_ex2_reach}

We consider the case where
\[
    A=\begin{bmatrix}-1&0\\0&-\sqrt{2}\end{bmatrix},
    \quad
    b_1=\begin{bmatrix}1\\-1\end{bmatrix},
    \quad
    b_2=\begin{bmatrix}1\\1\end{bmatrix},
    \quad
    B=\begin{bmatrix}1&1\\-1&1\end{bmatrix}.
\]
Note that $b_1$ and $b_2$ are the columns of $B$. Since all eigenvalues of $A$ are negative, it is stable.
Furthermore, one checks that $b_1$ and $Ab_1$
are linearly independent, hence $(A,b_1)$ is controllable, and similarly for $(A,b_2)$.
We can apply \Cref{thm:bdry_descrptn} to get the description of the boundary:
\begin{align*}
    \partial\Reach(A,b_1)
        &=\set*{\int_0^{\infty} e^{A t}b\sgn(c^Te^{A t}b)\dd t: c \in \RR^2\setminus \{\mathbf{0}\}}\\
        &=\Bigg\{\int_0^{\infty} \begin{bmatrix}e^{-t}&0\\0&e^{-\sqrt{2}t}\end{bmatrix}
            \begin{bmatrix}1\\1\end{bmatrix}\times\\
        &\hspace{0.5cm}\sgn\left(\begin{bmatrix}c_1&c_2\end{bmatrix}\begin{bmatrix}e^{-t}&0\\0&e^{-\sqrt{2}t}\end{bmatrix}
            \begin{bmatrix}1\\1\end{bmatrix}\right)\dd t: c \in \RR^2\setminus \{\mathbf{0}\}\Bigg\}\\
        &=\set*{\int_0^{\infty}\begin{bmatrix}e^{-t}\\e^{-\sqrt{2}t}\end{bmatrix}
            \sgn\left(c_1e^{-t}+c_2e^{-\sqrt{2}t}\right)\dd t: c \in \RR^2\setminus \{\mathbf{0}\}}.
\end{align*}
Given $c_1,c_2$ nonzero, observe that
\[
    c_1e^{-t}+c_2e^{-\sqrt{2}t}=0
    \;\Leftrightarrow\;
    1+\tfrac{c_2}{c_1}e^{(1-\sqrt{2})t}=0
    \;\Leftrightarrow\;
    t=\tfrac{1}{\sqrt{2}-1}\ln\tfrac{-c_2}{c_1}.
\]
Since only the ratio $c_2/c_1$ is important and $c_1$ must be nonzero, we can parametrize it as
$c_1=1$ and $c_2=-\alpha$, where $\alpha\in(1,+\infty)$, so that the only possible solution becomes
$t=t_1:=\tfrac{1}{\sqrt{2}-1}\ln \alpha\geqslant0$.
Now there are two cases:
\begin{itemize}
    \item if $c_2/c_1\geqslant -1$, then there is no solution and the integral becomes
        \[
            \int_0^{\infty}\begin{bmatrix}e^{-t}\\e^{-\sqrt{2}}\end{bmatrix}\sgn(c_1+c_2)\dd t
                =\begin{bmatrix}1\\\tfrac{1}{\sqrt{2}}\end{bmatrix}\sgn(c_1+c_2);
        \]
    \item  if $c_2/c_1<-1$, then we can split the integral into two parts (the sign must change):
        \begin{align*}
            \int_0^{t_1}&\begin{bmatrix}e^{-t}\\e^{-\sqrt{2}t}\end{bmatrix}\sgn(c_1+c_2)\dd t
                -\int_{t_1}^{\infty}\begin{bmatrix}e^{-t}\\e^{-\sqrt{2}t}\end{bmatrix}\sgn(c_1+c_2)\dd t\\
                &=\begin{bmatrix}1-e^{-t_1}\\\tfrac{1}{\sqrt{2}}-\sqrt{2}e^{-\sqrt{2}t_1}\end{bmatrix}\sgn(c_1+c_2)\\
                &=\begin{bmatrix}1-e^{-\tfrac{1}{\sqrt{2}-1}\ln\alpha}\\
                    \tfrac{1}{\sqrt{2}}-\sqrt{2}e^{-\tfrac{\sqrt{2}}{\sqrt{2}-1}\ln\alpha}\end{bmatrix}\sgn(c_1+c_2)\\
                &=\begin{bmatrix}1-2\alpha^{\tfrac{1}{1-\sqrt{2}}}\\
                    \tfrac{1}{\sqrt{2}}-\sqrt{2}\alpha^{\frac{\sqrt{2}}{1-\sqrt{2}}}\end{bmatrix}\sgn(c_1+c_2).
        \end{align*}
\end{itemize}
Hence the boundary of the reachable set is
\[
    \partial\Reach(A,b_1)=\set*{\pm\begin{bmatrix}1-2\alpha^{\tfrac{1}{1-\sqrt{2}}}\\
                    \tfrac{1}{\sqrt{2}}-\sqrt{2}\alpha^{\frac{\sqrt{2}}{1-\sqrt{2}}}\end{bmatrix}:\alpha\in[1,\infty)}
        \cup\set*{\pm\begin{bmatrix}1\\\tfrac{1}{\sqrt{2}}\end{bmatrix}}.
\]
The second reachable set is symmetric with respect to the vertical axis, hence we obtain
\[
    \partial\Reach(A,b_2)=\set*{\pm\begin{bmatrix}1-2\alpha^{\tfrac{1}{1-\sqrt{2}}}\\
                    \sqrt{2}\alpha^{\frac{\sqrt{2}}{1-\sqrt{2}}}-\tfrac{1}{\sqrt{2}}\end{bmatrix}:\alpha\in[1,\infty)}
        \cup\set*{\pm\begin{bmatrix}1\\-\tfrac{1}{\sqrt{2}}\end{bmatrix}}.
\]
In order to describe the border of $\partial\Reach(A,B)$, we consider the following
question: given a nonzero vector $\tau$, find the (unique by strict convexity) maximizer
in $\partial\Reach(A,b_1)$ in direction $\tau$:
\[
    x_\tau:=\argmax_{x\in \partial\Reach(A,b_1)}\Angle{x,\tau}.
\]
Observe that we have in fact already computed this point because the description of the border by
\Cref{thm:bdry_descrptn} is in the form of maximizers. If we write $\tau=(c_1,c_2)$ 
as above, then
\[
    x_\tau=\begin{cases}
        \begin{bmatrix}1-2(-c_2/c_1)^{\tfrac{1}{1-\sqrt{2}}}\\
            \tfrac{1}{\sqrt{2}}-\sqrt{2}(-c_2/c_1)^{\frac{\sqrt{2}}{1-\sqrt{2}}}\end{bmatrix}\sgn(c_1+c_2)
            &\text{ if }c_2/c_1\leqslant -1\\
        \begin{bmatrix}1\\\tfrac{1}{\sqrt{2}}\end{bmatrix}\sgn(c_1+c_2)&\text{ otherwise}.
        \end{cases}.
\]
A similar computation for $\partial\Reach(A,b_2)$ gives
\[
    x_\tau'=\begin{cases}
        \begin{bmatrix}1-2(c_2/c_1)^{\tfrac{1}{1-\sqrt{2}}}\\
            \sqrt{2}(c_2/c_1)^{\frac{\sqrt{2}}{1-\sqrt{2}}}-\tfrac{1}{\sqrt{2}}\end{bmatrix}\sgn(c_1-c_2)
            &\text{ if }c_2/c_1\geqslant 1\\
        \begin{bmatrix}1\\-\tfrac{1}{\sqrt{2}}\end{bmatrix}\sgn(c_1-c_2)
            &\text{ otherwise}
        \end{cases}.
\]
Noting that $\sgn(c_1+c_2)=-\sgn(c_1-c_2)$ whenever $|c_2/c_1|\geqslant1$, 
the maximizer $y_\tau$ for $\partial\Reach(A,B)$ is
\begin{align*}
    y_\tau=x_\tau+x_\tau'
        &=\begin{cases}
            \begin{bmatrix}2-2(-c_2/c_1)^{\tfrac{1}{1-\sqrt{2}}}\\
                -\sqrt{2}(-c_2/c_1)^{\frac{\sqrt{2}}{1-\sqrt{2}}}\end{bmatrix}\sgn(c_1+c_2)
            &\text{if }c_2/c_1\leqslant-1\\
            \begin{bmatrix}2-2(c_2/c_1)^{\tfrac{1}{1-\sqrt{2}}}\\
                \sqrt{2}(c_2/c_1)^{\frac{\sqrt{2}}{1-\sqrt{2}}}\end{bmatrix}\sgn(c_1+c_2)
            &\text{otherwise}
        \end{cases}\\
        &=\begin{bmatrix}-2|c_2/c_1|^{\tfrac{1}{1-\sqrt{2}}}\\
            \sqrt{2}\left(1-|c_2/c_1|^{\frac{\sqrt{2}}{1-\sqrt{2}}}\right)\sgn(c_1/c_2)\end{bmatrix}\sgn(c_1+c_2)
\end{align*}
for $|c_2/c_1|\geqslant1$. 

}{}

\end{document}